\documentclass[12pt]{amsart}
\usepackage{tikz,hyperref,pgf,pgflibraryarrows,enumerate}
\usepackage{cite} 

\newtheorem{thm}{Theorem}[section]
\newtheorem{cor}[thm]{Corollary}
\newtheorem{lem}[thm]{Lemma}
\newtheorem{prop}[thm]{Proposition}

\theoremstyle{definition}
\newtheorem{dfn}[thm]{Definition}
\newtheorem{dfns}[thm]{Definitions}

\theoremstyle{remark}
\newtheorem{rmk}[thm]{Remark}
\newtheorem{rmks}[thm]{Remarks}
\newtheorem{example}[thm]{Example}
\newtheorem{examples}[thm]{Examples}

\newcommand{\N}{\mathbb{N}}
\newcommand{\Z}{\mathbb{Z}}

\newcommand{\NN}{\N}

\newcommand{\ZZ}{\Z}
\newcommand{\TT}{\mathbb{T}}

\def\clsp{\operatorname{\overline{span}}}

\def\sp{\operatorname{sp}}

\title{Simplicity of the $C^*$-algebras of skew product $k$-graphs}

\author{Ben Maloney and David Pask}
\address{School of Mathematics \& Applied Statistics \\ University
of Wollongong \\ Northfields Avenue \\ NSW 2522 \\ AUSTRALIA} \email{bkm611@uowmail.edu.au, dpask@uow.edu.au}



\date{\today}

\begin{document}

\begin{abstract}
We consider conditions on a $k$-graph $\Lambda$, a semigroup $S$ and a functor $\eta : \Lambda \to S$ which ensure that the $C^*$-algebra of the skew-product graph $\Lambda \times_\eta S$ is simple. Our results allow give some necessary and sufficient conditions for the AF-core of a $k$-graph $C^{*}$-algebra to be simple.
\end{abstract}

\maketitle

\section{Introduction}

In \cite{RobertsonSteger} Robertson and Steger investigated $C^*$-algebras which they considered to be higher-rank versions of the Cuntz-Krieger algebras. Subsequently, in \cite{KP} Kumjian and Pask introduced higher-rank graphs, or $k$-graphs, as a graphical means to provide combinatorial models for the Cuntz-Krieger algebras of Robertson and Steger. They showed how to construct a $C^{*}$-algebra that is associated to a $k$-graph. Since then $k$-graphs and their $C^*$-algebras have attracted a lot of attention from many authors (see \cite{ALRS,DY,DPY,E,KP,LewinSims,MPR,PQR,PRW,PRho,Po,RSY,RobS}).

Roughly speaking, a $k$-graph is a category $\Lambda$ together with a functor $d : \Lambda \to \NN^k$ satisfying a certain factorisation property. A $1$-graph is then the path category of a directed graph. Given a functor $\eta : \Lambda \to S$, where $S$ is a semigroup with identity, we may form the skew product $k$-graph $\Lambda \times_\eta S$. Skew product graphs play an important part in the development of $k$-graph $C^*$-algebras. For example \cite[Corollary 5.3]{KP} shows that $C^* ( \Lambda \times_d \ZZ^k )$ is isomorphic to $C^* ( \Lambda ) \times_\gamma \TT^k$ where $\gamma : \TT^k \to \operatorname{Aut} C^* ( \Lambda )$ is the canonical gauge action. Skew product graphs feature in nonabelian duality: In \cite{MPR} it is shown that if a right-reversible semigroup (Ore semigroup) $S$ acts freely on a $k$-graph $\Lambda$ then the crossed product $C^* ( \Lambda ) \times S$ is stably isomorphic to $C^* ( \Lambda / S)$. On the other hand if $S$ is a group $G$ then $C^* ( \Lambda \times_\eta G)$ is isomorphic to the crossed product $C^* ( \Lambda ) \times_{\delta_\eta} G$ where $\delta_\eta$ is the coaction of $G$ on $C^* ( \Lambda )$ induced by $\eta$.

The main purpose of this paper is to investigate necessary and sufficient conditions for the $C^*$-algebra of a skew product $k$-graph to be simple. We will be particularly interested in the specific case when $S = \NN^k$ and $\eta=d$. It can be shown that simplicity of $C^* ( \Lambda \times_d \NN^k )$ is equivalent to simplicity of the fixed point algebra (AF core) $C^* ( \Lambda )^\gamma$.
This is important as many results in the literature apply particularly when AF core is simple; see \cite{ALRS}.

We begin by introducing some basic facts we will need during this paper.

\section{Background}

\subsection{Basic facts about $k$-graphs} \label{sec:bg}

All semigroups in this paper will be countable, cancellative and have an identity, hence any semigroup may be considered as a category with a single object. The semigroup $\NN^k$ is freely generated by $\{ e_1 , \ldots , e_k \}$ and comes with the usual order structure: if $n = \sum_{i=1}^k n_i e_i$ and $m = \sum_{i=1}^k m_i e_i$ then $m > n$ (resp.\ $m \ge n$) if $m_i > n_i$ (resp.\ $m_i \ge n_i$) for all $i$. For $m,n \in \NN^k$ we define $m \vee n \in \NN^k$ by $( m \vee n )_i= \max \{m_i , n_i \}$ for $i=1 , \ldots , k$.

A directed graph $E$ is a quadruple $(E^0,E^1,r,s)$ where $E^0,E^1$ are countable sets of vertices and edges. The direction of an edge $e \in E^1$ is given by the maps $r,s : E^1 \to E^0$. A path $\lambda$ of length $n \ge 1$ is a sequence $\lambda = \lambda_1 \cdots \lambda_n$ of edges such that $s(\lambda_i)=r(\lambda_{i+1})$ for $i=1,\ldots,n-1$. The set of paths of length $n \ge 1$ is denoted $E^n$. We may extend $r,s$ to $E^n$ for $n \ge 1$ by $r ( \lambda ) = r ( \lambda_1 )$ and $s(\lambda)=s(\lambda_n)$ and to $E^0$ by $r(v)=v=s(v)$.

A \textit{higher-rank graph} or \textit{$k$-graph} is a combinatorial structure, and is a $k$-dimensional analogue of a directed graph. A $k$-graph consists of a countable category $\Lambda$ together with a functor $d : \Lambda \rightarrow \NN^{k}$, known as the degree map, with the following factorisation property: for every morphism $\lambda \in \Lambda$ and every decomposition $d (\lambda) = m+n$, there exist unique morphisms $\mu, \nu \in \Lambda$ such that $d (\mu) = m$, $d (\nu) = n$, and $\lambda = \mu \nu$.

For $n \in \NN^k$ we define $\Lambda^n := d^{-1} (n)$ to be those morphisms in $\Lambda$ of degree $n$. Then by the factorisation property $\Lambda^{0}$ may be identified with the objects of $\Lambda$, and are called vertices. For $u,v \in \Lambda^{0}$, $ X \subseteq \Lambda$ and $n \in \NN^k$ we set
\begin{equation*}
u X = \{ \lambda \in X : r ( \lambda ) = u \} \quad X v = \{ \lambda \in X : s ( \lambda ) = v \} \quad
u X v = uX \cap Xv .
\end{equation*}

\noindent
A $k$-graph $\Lambda$ is visualised by a $k$-coloured directed graph $E_\Lambda$ with vertices $\Lambda^0$ and edges $\sqcup_{i=1}^k \Lambda^{e_i}$ together with range and source maps inherited from $\Lambda$ called its $1$-skeleton. The $1$-skeleton is provided with square relations $\mathcal{C}_\Lambda$ between the edges in $E_\Lambda$, called \textit{factorisation rules}, which come from factorisations of morphisms in $\Lambda$ of degree $e_i + e_j$ where $i \neq j$. By convention the edges of degree $e_1$ are drawn blue (solid) and the edges of degree $e_2$ are drawn red (dashed). For more details about the $1$-skeleton of a $k$-graph see \cite{RSY}. On the other hand, if $G$ is a $k$-coloured directed graph with a complete, associative collection of square relations $\mathcal{C}$ completely determines a $k$-graph $\Lambda$ such that $E_\Lambda=G$ and $\mathcal{C}_\Lambda=\mathcal{C}$ (see \cite{HRSW}).

A $k$-graph $\Lambda$ is \textit{row-finite} if for every $v\in\Lambda^{0}$ and every $n\in\N^{k}$, $v\Lambda^{n}$ is finite. A $k$-graph has \emph{no sources} if $v \Lambda^n \neq \emptyset$ for all $v \in \Lambda^0$ and nonzero $n \in \NN^k$. A $k$-graph has \emph{no sinks} is $\Lambda^n v \neq \emptyset$ for all $v \in \Lambda^0$ and nonzero $n \in \NN^k$.

For $\lambda \in \Lambda$ and $m \leq n \leq d ( \lambda )$, we define $\lambda ( m , n )$ to be the unique path in $\Lambda^{ n - m }$ obtained from the $k$-graph factorisation property such that $\lambda = \lambda^{ \prime } ( \lambda ( m , n ) ) \lambda^{ \prime \prime }$ for some $\lambda^{ \prime } \in \Lambda^{m}$ and $\lambda^{ \prime \prime }\in \Lambda^{ d (\lambda) - n}$.

\begin{examples} \label{ex:kex}
\begin{enumerate}[(a)]
\item In \cite[Example 1.3]{KP} it is shown that the path category $E^*=\cup_{i \ge 0} E^i$ of a directed graph $E$ is a $1$-graph, and vice versa. For this reason we shall move seamlessly between $1$-graphs and directed graphs.
\item For $k \ge 1$ let $T_k$ be the category with a single object $v$ and generated by $k$ commuting morphisms $\{ f_1 , \ldots , f_k \}$. Define $d : T_k \to \NN^k$ by $d ( f_1^{n_1} \ldots f_k^{n_k} ) = ( n_1, \ldots , n_k )$ then it is straightforward to check that $T_k$ is a $k$-graph. We frequently identify $T_k$ with $\NN^k$ via the map $f_1^{n_1} \cdots f_k^{n_k} \mapsto ( n_1 , \ldots , n_k )$.
\item For $k \ge 1$ define a category $\Delta_k$ as follows: Let $\operatorname{Mor} \Delta_k = \{ (m,n) \in \mathbb{Z}^k \times \mathbb{Z}^k : m \le n \}$ and $\operatorname{Obj} \Delta_k = \mathbb{Z}^k$; structure maps $r (m,n) = m$, $s (m,n) = n$, and composition $(m,n)(n,p)=(m,p)$. Define $d : \Delta_k \to \NN^k$ by $d (m,n) = n-m$, then one checks that $( \Delta_k , d )$ is a row-finite $k$-graph. We identify $\operatorname{Obj} \Delta_k$ with $\{ (m,m) : m \in \mathbb{Z}^k \} \subset \operatorname{Mor} \Delta_k$.
\item For $n \geq 1$ let $\underline{n} = \{ 1 , \ldots , n \}$. For $m,n \geq 1$ let $\theta : \underline{m} \times \underline{n} \to \underline{m} \times \underline{n}$ a bijection. Let $\mathbb{F}^2_\theta$ be the $2$-graph which has $1$-skeleton which consists of with single vertex $v$ and edges $f_1 , \ldots , f_m , g_1 , \ldots , g_n$, such that $f_i$ have the same colour (blue) for $i \in \underline{m} $ and $g_j$ have the same colour (red) for $j \in \underline{n}$ together with complete associative square relations $f_i g_j = g_{j'} f_{i'}$ where $\theta ( i , j ) = ( i' , j' )$ for $(i,j) \in \underline{m} \times \underline{n}$ (for more details see \cite{DPY,DY,Po}).
\end{enumerate}
\end{examples}


\subsection{Skew product $k$-graphs}

Let $\Lambda$ be a $k$-graph and $\eta:\Lambda\to S$ a functor into a semigroup $S$. We can make the cartesian product $\Lambda\times S$ into a $k$-graph $\Lambda\times_\eta S$ by taking $(\Lambda\times_\eta S)^0=\Lambda^0\times S$, defining $r,s:\Lambda\times_\eta S\to (\Lambda\times_\eta S)^0$ by
\begin{equation} \label{eq:rsdef}
r(\lambda,t)=(r(\lambda),t)\ \text{ and }\ s(\lambda,t)=(s(\lambda),t\eta(\lambda)),
\end{equation}
defining the composition by
\[
(\lambda,t)(\mu, u)=(\lambda\mu,t)\ \text{ when $s(\lambda,t)=r(\mu, u)$ (so that $u=t\eta(\lambda)$\,),}
\]
and defining $d:\Lambda\times_\eta S\to \N^k$ by $d(\lambda,t)=d(\lambda)$. As in \cite{MPR} it is straightforward to show that this defines a $k$-graph.


\begin{rmk} If $\Lambda$ is row-finite with no sources and $\eta : \Lambda \to S$ a functor
then $\Lambda \times_\eta S$ is row-finite with no sources.
\end{rmk}

A $k$-graph morphism is a degree preserving functor between two $k$-graphs. If a $k$-graph morphism is bijective, then it is called an isomorphism.

\begin{examples} \label{ex:pi}
\begin{enumerate}[(i)]
\item Let $\Lambda$ be a $k$-graph and $\eta : \Lambda \to S$ a functor, where $S$ is a semigroup and $\Lambda \times_\eta S$ the associated skew product graph. Then the map $\pi : \Lambda \times_\eta S \to \Lambda$ given by $\pi ( \lambda , s ) = \lambda$ is a surjective $k$-graph morphism.
\item For $k \geq 1$ the map $( k , m ) \mapsto (m,m+k)$ gives an isomorphism from $T_k \times_d \ZZ^k$ to $\Delta_k$.
\end{enumerate}
\end{examples}

\begin{dfn} \label{dfn:path-lift}
Let $\Lambda, \Gamma$ be row-finite $k$-graphs. A surjective $k$-graph morphism $p : \Lambda \to \Gamma$ has \textit{$r$-path lifting} if for all $v \in \Lambda^0$ and $\lambda \in p(v) \Gamma$ there is $\lambda' \in v \Lambda$ such that $p ( \lambda' ) = \lambda$. If $\lambda'$ is the unique element with this property then $p$ has \textit{unique $r$-path lifting}.
\end{dfn}

\begin{example} \label{ex:rlift}
Let $\Lambda$ be a row-finite $k$-graph and $\eta : \Lambda \to S$ a functor where $S$ is a semigroup, and $\Lambda \times_\eta S$ the associated skew product graph. The map $\pi : \Lambda \times_\eta S \to \Lambda$ described in Examples~\ref{ex:pi}(i) has unique $r$-path lifting.
\end{example}

\subsection{Connectivity} \label{ss:conn}

A $k$-graph $\Lambda$ is \textit{connected} if the equivalence relation on $\Lambda^0$ generated by the relation $\{ (u,v) : u \Lambda v \neq \emptyset \}$ is $\Lambda^0 \times \Lambda^0$. The $k$-graph $\Lambda$ is \textit{strongly connected} if for all $u, v \in \Lambda^0$ there is $N > 0$ such that $u \Lambda^N v \neq \emptyset$. If $\Lambda$ is strongly connected, then it is connected and has no sinks or sources.
The $k$-graph $\Lambda$ is \emph{primitive} is there is $N > 0$ such that $u \Lambda^N v \neq \emptyset$ for all $u,v \in \Lambda^0$. If $\Lambda$ is primitive then it is strongly connected.

\begin{examples}
The graphs $T_k$ and $\mathbb{F}_\theta^2$ from Examples~\ref{ex:kex} are primitive since they only have one vertex.
\end{examples}

The connectivity of a $k$-graph may also be described in terms of its component matrices as defined in \cite[\S 6]{KP}:
Given a $k$-graph $\Lambda$, for $1\leq i\leq k$ and $u,v\in\Lambda^{0}$, we define $k$ non-negative $\Lambda^0 \times \Lambda^0$ matrices $M_{i}$ with entries $M_{i}(u,v)=|u\Lambda^{e_{i}}v|$. Using the $k$-graph factorisation property, we have that $|u\Lambda^{e_{i}+e_{j}}v|=|u\Lambda^{e_{j}+e_{i}}v|$ for all $u,v \in \Lambda^0$, and so $M_{i}M_{j}=M_{j}M_{i}$. For $m=(m_{1},\ldots,m_{k})\in\N^{k}$ and $u,v \in \Lambda^0$, we have $|u\Lambda^{m}v|=(M_{1}^{m_{1}} \cdots M_{k}^{m_{k}})(u,v)=M^{m}(u,v)$, using multiindex notation. The following lemma follows directly from the definitions given above.

\begin{lem} \label{lem:matrixcond}
Let $\Lambda$ be a row-finite $k$-graph with no sources.
\begin{enumerate}[(a)]
\item Then $\Lambda$ is strongly connected if and only if for all pairs $u,v\in\Lambda^{0}$ there is there is $N \in \mathbb{N}^k$ such that $M^{N}(u,v) > 0$.
\item Then $\Lambda$ is primitive if and only if there is $N >0$ such that $M^{N}(u,v) > 0$ for all pairs $u,v\in\Lambda^{0}$.
\end{enumerate}
\end{lem}

\begin{rmks} \label{rmks:noss}
Following \cite[\S 4]{PRho}, a primitive $1$-graph $\Lambda$ is strongly connected with period $1$; that is, the greatest common divisor of all $n$ such that $ v \Lambda^n v$ for some $v\in \Lambda^0$ is $1$.
\end{rmks}

\begin{lem} \label{lemma:cofinal-loops}
Let $\Lambda$ be a $k$-graph with no sinks, and $\Lambda^{0}$ finite. Then for all $v\in\Lambda^{0}$, there exists $w\in\Lambda^{0}$ and $\alpha\in w\Lambda w$ such that $d(\alpha)>0$ and $w\Lambda v\neq\emptyset$.
\end{lem}

\begin{proof}
Let $p=(1,\ldots,1)\in\N^{k}$. Since $v$ is not a sink, there exists $\beta_{1}\in\Lambda^{p}v$. Since $r(\beta_{1})$ is not a sink, there exists $\beta_{2}\in\Lambda^{p}r(\beta_{1})$. Inductively, there exist infinitely many $\beta_{i}$ such that $d(\beta_{i})=p$ and $r(\beta_{i})=s(\beta_{i+1})$. Since $\Lambda^{0}$ is finite, there exists $w\in\Lambda^{0}$ such that $r(\beta_{i})=w$ for infinitely many $i$. Suppose $r(\beta_{n})=w=r(\beta_{m})$ with $m>n$. Then $\alpha=\beta_{m}\ldots\beta_{n+1}$ has the requisite properties, and $w\Lambda v\neq\emptyset$, since $\beta_{n}\ldots\beta_{1}\in w\Lambda v$.
\end{proof}

\subsection{The graph $C^{*}$-algebra} \label{ss:graph-alg}

Let $\Lambda$ be a row-finite $k$-graph with no sources, then following \cite{KP}, a \textit{Cuntz-Krieger $\Lambda$-family} in a $C^*$-algebra $B$ consists of partial isometries $\{S_\lambda:\lambda\in\Lambda\}$ in $B$ satisfying the \textit{Cuntz-Krieger relations}:
\begin{itemize}\item[]\begin{itemize}
\item[(CK1)] $\{S_v:v\in \Lambda^0\}$ are mutually orthogonal projections;
\item[(CK2)] $S_{\lambda}S_{\mu}=S_{\lambda\mu}$ whenever $s(\lambda)=r(\mu)$;
\item[(CK3)] $S_\lambda^*S_\lambda=S_{s(\lambda)}$ for every $\lambda\in\Lambda$;
\item[(CK4)] $S_v=\sum_{\{\lambda\in v\Lambda^{n}\}}S_\lambda S_\lambda^*$ for every $v\in \Lambda^0$ and $n\in \N^k$.
\end{itemize}\end{itemize}
The $k$-graph $C^{*}$-algebra $C^*(\Lambda)$ is generated by a universal Cuntz-Krieger $\Lambda$-family $\{s_\lambda\}$. By \cite[Proposition~2.11]{KP} there exists a Cuntz-Krieger $\Lambda$-family such that each vertex projection $S_v$ (and hence by (CK3) each $S_{\lambda}$) is nonzero and so there exists a nonzero universal $k$-graph $C^{*}$-algebra for a Cuntz-Krieger $\Lambda$-family. Moreover,
\[
C^*(\Lambda)=\clsp\{s_\lambda s_\mu^*:\lambda,\mu\in\Lambda, s(\lambda)=s(\mu)\}\ \text{ (see \cite[Lemma~3.1]{KP}).}
\]

\noindent
We will use~\cite[Theorem 3.1]{RobS} by Robertson and Sims when considering the simplicity of graph $C^{*}$-algebras:

\begin{thm}[Robertson-Sims] \label{thm:rs}
Suppose $\Lambda$ is a row-finite $k$-graph with no sources. Then $C^{*}(\Lambda)$ is simple if and only if $\Lambda$ is cofinal and aperiodic.
\end{thm}

\noindent
We now focus on the two key properties involved in the simplicity criterion of Theorem \ref{thm:rs}, namely aperiodicity and cofinality. Our attention will be directed towards applying these conditions on skew product graphs.

\section{Aperiodicity}

Our definition of aperiodicity is taken from Robertson-Sims,~\cite[Theorem 3.2]{RobS}.

\begin{dfns} A row-finite $k$-graph $\Lambda$ with no sources has \textit{no local periodicity} at $v\in\Lambda^{0}$ if for all $m\neq n\in\NN^{k}$ there exists a path $\lambda\in v\Lambda$ such that $d(\lambda)\geq m\vee n$ and
\[
\lambda(m,m+d(\lambda)-(m\vee n))\neq\lambda(n,n+d(\lambda)-(m\vee n)).
\]

\noindent $\Lambda$ is called \emph{aperiodic} if every $v \in \Lambda^0$ has no local periodicity.
\end{dfns}

\begin{examples} \label{ex:aperiodex}
\begin{enumerate}[(a)]
\item The $k$-graph $\Delta_k$ is aperiodic for all $k \geq 1$. First observe that there is no local periodicity at $v=(0,0)$. Given $m \neq n \in \mathbb{N}^k$, let $N \geq m \vee n$; then $\lambda = (0,N)$ is the only element of $v \Delta_k$. Then $\lambda(m,m) = (m,m) \neq (n,n) = \lambda(n,n)$. A similar argument applies for any other vertex $w= (n,n)$ in $\Delta_k$ and so there is no local periodicity at $w$ for all $w \in \Delta_k^0$.
\item The $k$-graph $T_k$ is not aperiodic for all $k \geq 1$. For all $n \in \NN^k$ one checks that $f_1^{n_1} \cdots f_k^{n_k}$ is the only element of $v T_k^n$. Hence given $m \neq n \in \mathbb{N}^k$ it follows that for all $\lambda \in v \Lambda^{N}$ with $N \geq m \vee n$ we have
\[
\lambda (m,m+(m \vee n )) = \lambda ( n , n + ( m \vee n ) ) .
\]
\end{enumerate}
\end{examples}

%

\noindent
Since the map $\pi : \Lambda \times_\eta S \to \Lambda$ has unique $r$-path lifting, we wish to know if we can deduce the aperiodicity of $\Lambda \times_\eta S$ from that of $\Lambda$. A corollary of our main result Theorem~\ref{thm:lifts2}, shows that this is true.

\begin{thm} \label{thm:lifts2}
Let $\Lambda, \Gamma$ be row-finite $k$-graphs and $p : \Lambda \to \Gamma$ have $r$-path lifting. If $\Gamma$ is aperiodic, then $\Lambda$ is aperiodic.
\end{thm}

\begin{proof}
Suppose that $\Gamma$ is aperiodic. Let $v \in \Lambda^{0}$ and $m\neq n\in\N^{k}$. Since $\Gamma$ is aperiodic, there exists $\lambda \in p(v) \Gamma$ with $d(\lambda)\geq m \vee n$ such that $\lambda(m,m+d(\lambda)-(m\vee n))\neq\lambda(n,n+d(\lambda)-(m\vee n))$. By $r$-path lifting there is $\lambda' \in v \Lambda$ with $p ( \lambda' ) = \lambda$ such that $d ( \lambda' ) \geq m \vee n$ and
\[
\lambda' (m,m+d(\lambda)-(m\vee n)) \neq \lambda'(n,n+d(\lambda)-(m\vee n)),
\]

\noindent
and so $\Lambda$ is aperiodic.
\end{proof}

\noindent
The converse of Theorem \ref{thm:lifts2} is false:

\begin{example} \label{ex:notconv}
The surjective $k$-graph morphism $p:\Delta_k \rightarrow T_k$ given by $p(m,m+e_i)=f_i$ for all $m \in \ZZ^k$ and $i=1,\ldots,k$ has $r$-path lifting. However by Examples~\ref{ex:aperiodex} we see that $\Delta_k \cong T_k \times_d \ZZ^k$ is aperiodic but $T_k$ is not.
\end{example}

\begin{cor} \label{cor:spg1}
Let $\Lambda$ be a row-finite $k$-graph with no sources, $\eta : \Lambda \to S$ a functor where $S$ is a semigroup and $\Lambda \times_\eta S$ the associated skew product graph. If $\Lambda$ is aperiodic then $\Lambda \times_\eta S$ is aperiodic.
\end{cor}

\begin{proof}
Follows from Theorem \ref{thm:lifts} and Example \ref{ex:rlift}.
\end{proof}

\noindent In some cases the aperiodicity of a skew product graph $\Lambda \times_\eta S$ can be deduced directly from properties of $\eta$.

\begin{prop} \label{prop:semigp-aperiodic}
Suppose $S$ is a semigroup, $\Lambda$ is a row-finite $k$-graph, $\eta:\Lambda\rightarrow S$ is a functor, and there exists a map $\phi:S\rightarrow\ZZ^{k}$ such that $d = \phi \circ \eta$. Then $\Lambda\times_{\eta}S$ is aperiodic.
\end{prop}

\begin{proof}
Fix $(v,s)\in(\Lambda\times_{\eta}S)^{0}$ and $m \neq n \in \NN^{k}$. Let $\lambda \in (v,s)(\Lambda\times_{\eta}S)$ be such that $d(\lambda)\geq m \vee n$. Observe that $\lambda(m,m)=s(\lambda(0,m))$, $\lambda(m,m)$ is of the form $(w,s\eta(\lambda(0,m)))$ for some $w\in\Lambda^{0}$. Similarly, $\lambda(n,n)$ is of the form $(w^{\prime}, s\eta(\lambda(0,n)))$ for some $w^{\prime}\in \Lambda^{0}$.

We claim $\lambda(m,m)\neq\lambda(n,n)$: Suppose, by hypothesis, $\eta(\lambda(0,n))=\eta(\lambda(0,m))$. Then $n = \phi \circ \eta(\lambda(0,n)) = \phi \circ \eta(\lambda(0,m)) =m$, which provides a contradiction, and $m\neq n$. Then $\eta(\lambda(0,m))\neq\eta(\lambda(0,n))$, and so $\lambda(m,m) \neq \lambda(n,n)$, and hence $\lambda (m, m + d(\lambda) -(m\vee n)) \neq \lambda(n,n+d(\lambda)-(m\vee n))$.
\end{proof}

\begin{cor}
Suppose $\Lambda$ is a row-finite $k$-graph. Then $\Lambda\times_{d}\NN^k$ and $\Lambda\times_{d}\ZZ^k$ are aperiodic.
\end{cor}

\begin{proof}
Apply Proposition~\ref{prop:semigp-aperiodic} with $\eta=d$ and $S=\NN, \ZZ$ respectively.
\end{proof}


\section{Cofinality}

We will use the Lewin-Sims definition of cofinality,~\cite[Remark A.3]{LewinSims}:

\begin{dfn} A row-finite, $k$-graph $\Lambda$ with no sources is \textit{cofinal} if for all pairs $v,w\in\Lambda^{0}$ there exists $N\in\N^{k}$ such that $v\Lambda s(\alpha)\neq\emptyset$ for every $\alpha\in w\Lambda^{N}$.
\end{dfn}


\begin{lem} \label{lem:cofinalconn}
Let $\Lambda$ be a row-finite $k$-graph with no sources.
\begin{enumerate}[(a)]
\item If $\Lambda$ is cofinal then $\Lambda$ is connected.
\item Suppose that for all pairs $v,w\in\Lambda^{0}$ there exists $N\in\N^{k}$ such that $v\Lambda s(\alpha)\neq\emptyset$ for every $\alpha\in w\Lambda^{N}$. Then for $n \ge N$
  we have $v\Lambda s(\alpha)\neq\emptyset$ for every $\alpha\in w\Lambda^{n}$.
\end{enumerate}
\end{lem}

\begin{proof}
Fix $v,w \in \Lambda^0$. If $\Lambda$ is cofinal it follows that there is $\alpha \in w \Lambda$ such that $w \Lambda s ( \alpha ) $ and $v \Lambda s ( \alpha )$ are non-empty. It then follows that $(v,w)$ belongs to the equivalence relation described in Section~\ref{ss:conn}. Since $v,w$ were arbitrary it follows that the equivalence relation is $\Lambda^0 \times \Lambda^0$ and so $\Lambda$ is connected.

Fix $v , w \in \Lambda^0$, then there is $N\in\N^{k}$ such that $v \Lambda s(\alpha)\neq\emptyset$ for every $\alpha\in w\Lambda^{N}$. Let $n \geq N$ and consider $\beta \in w \Lambda^n$ then $\beta' = \beta ( 0 , N ) \in w \Lambda^N$ and so by hypothesis there is $\lambda \in v \Lambda s ( \beta' )$. Then $\lambda \beta ( N , n ) \in v \Lambda s ( \beta )$ and the result follows.
\end{proof}

\begin{lem} \label{lem:skelcofinal}
Let $\Lambda$ be a row-finite $k$-graph with skeleton $E_\Lambda$. If $E_\Lambda$ is cofinal then $\Lambda$ is cofinal. Furthermore, $\Lambda$ is strongly connected if and only if $E_\Lambda$ strongly connected
\end{lem}

\begin{proof}
Fix $v,w \in \Lambda^0 = E_\Lambda^0$. Since $E_\Lambda$ is cofinal there is $n \in \NN$ such that $v E_\Lambda s ( \alpha ) \neq \emptyset$ for all $\alpha \in w E_\Lambda^n$. Let $N \in \NN^k$ be such that
$\sum_{i=1}^k N_i=n$. Then for all $\alpha' \in w \Lambda^N$ we have $\alpha' \in E_\Lambda^n$ and so $v \Lambda^N s ( \alpha' ) \neq \emptyset$.

Suppose that $\Lambda$ is strongly connected and $v,w \in E_\Lambda^0 = \Lambda^0$. Since $\Lambda$ is strongly connected there is $\alpha \in v \Lambda w$ with $d ( \alpha ) > 0$. Let $n = \sum_{i=1}^n d(\alpha)_i$ then $n > 0$ and $v E_\Lambda w \neq \emptyset$, so $E_\Lambda$ is strongly connected.
Suppose that $E_\Lambda$ is strongly connected, and $v , w \in \Lambda^0 = E_\Lambda^0$. Since $\Lambda$ has no sources, there is a path $\alpha \in v E_\Lambda^k$ which uses an edge of each of the $k$-colours. Let $u = s ( \alpha )$. Since $E_\Lambda$ is strongly connected there is $\beta \in u E_\Lambda^n w$.
Let $\lambda$ be the element of $\Lambda$ which may be represented by $\alpha \beta \in E_\Lambda$. Then $\lambda \in v \Lambda w$ and $d ( \lambda ) > 0$ and so $\Lambda$ is strongly connected.
\end{proof}

\begin{rmk}
The converse to the first part of Lemma~\ref{lem:skelcofinal} is not true: Let $\Lambda$ be the $2$-graph which is completely determined by its $1$-skeleton as shown:
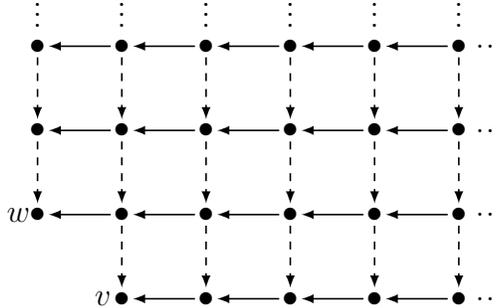
\[
\begin{tikzpicture}[scale=0.28]
          \node[circle,inner sep=0pt] at (-0.9, 4) {$w$};
				 \node[circle,inner sep=0pt] at (3.1, 0) {$v$};

          \node[circle,inner sep=0pt] (p12) at (0, 4) {$\bullet$};

          \node[circle,inner sep=0pt] (p13) at (0, 8) {$\bullet$};

          \node[circle,inner sep=0pt] (p14) at (0, 12) {$\bullet$};

           \draw[style=semithick, dashed, -latex] (p13.south)--(p12.north);
           \draw[style=semithick, dashed, -latex] (p14.south)--(p13.north);
          \node[circle,inner sep=0pt] (p21) at (4, 0) {$\bullet$};

          \node[circle,inner sep=0pt] (p22) at (4, 4) {$\bullet$};

          \node[circle,inner sep=0pt] (p23) at (4, 8) {$\bullet$};

          \node[circle,inner sep=0pt] (p24) at (4, 12) {$\bullet$};

           \draw[style=semithick, dashed, -latex] (p22.south)--(p21.north);
           \draw[style=semithick, dashed, -latex] (p23.south)--(p22.north);
           \draw[style=semithick, dashed, -latex] (p24.south)--(p23.north);
%
          \draw[style=semithick, -latex] (p22.west)--(p12.east);
          \draw[style=semithick, -latex] (p23.west)--(p13.east);
          \draw[style=semithick, -latex] (p24.west)--(p14.east);
          \node[circle,inner sep=0pt] (p31) at (8, 0) {$\bullet$};

          \node[circle,inner sep=0pt] (p32) at (8, 4) {$\bullet$};

          \node[circle,inner sep=0pt] (p33) at (8, 8) {$\bullet$};

          \node[circle,inner sep=0pt] (p34) at (8, 12) {$\bullet$};

           \draw[style=semithick, dashed, -latex] (p32.south)--(p31.north);
           \draw[style=semithick, dashed, -latex] (p33.south)--(p32.north);
           \draw[style=semithick, dashed, -latex] (p34.south)--(p33.north);
           \draw[style=semithick, -latex] (p31.west)--(p21.east);
          \draw[style=semithick, -latex] (p32.west)--(p22.east);
          \draw[style=semithick, -latex] (p33.west)--(p23.east);
          \draw[style=semithick, -latex] (p34.west)--(p24.east);
          \node[circle,inner sep=0pt] (p41) at (12, 0) {$\bullet$};

          \node[circle,inner sep=0pt] (p42) at (12, 4) {$\bullet$};

          \node[circle,inner sep=0pt] (p43) at (12, 8) {$\bullet$};

          \node[circle,inner sep=0pt] (p44) at (12, 12) {$\bullet$};

           \draw[style=semithick, dashed, -latex] (p42.south)--(p41.north);
           \draw[style=semithick, dashed, -latex] (p43.south)--(p42.north);
           \draw[style=semithick, dashed, -latex] (p44.south)--(p43.north);
           \draw[style=semithick, -latex] (p41.west)--(p31.east);
          \draw[style=semithick, -latex] (p42.west)--(p32.east);
          \draw[style=semithick, -latex] (p43.west)--(p33.east);
          \draw[style=semithick, -latex] (p44.west)--(p34.east);
          \node[circle,inner sep=0pt] (p51) at (16, 0) {$\bullet$};

          \node[circle,inner sep=0pt] (p52) at (16, 4) {$\bullet$};

          \node[circle,inner sep=0pt] (p53) at (16, 8) {$\bullet$};

          \node[circle,inner sep=0pt] (p54) at (16, 12) {$\bullet$};

           \draw[style=semithick, dashed, -latex] (p52.south)--(p51.north);
           \draw[style=semithick, dashed, -latex] (p53.south)--(p52.north);
           \draw[style=semithick, dashed, -latex] (p54.south)--(p53.north);
           \draw[style=semithick, -latex] (p51.west)--(p41.east);
          \draw[style=semithick, -latex] (p52.west)--(p42.east);
          \draw[style=semithick, -latex] (p53.west)--(p43.east);
          \draw[style=semithick, -latex] (p54.west)--(p44.east);
          \node[circle,inner sep=0pt] (p61) at (20, 0) {$\bullet$};

          \node[circle,inner sep=0pt] (p62) at (20, 4) {$\bullet$};

          \node[circle,inner sep=0pt] (p63) at (20, 8) {$\bullet$};

          \node[circle,inner sep=0pt] (p64) at (20, 12) {$\bullet$};

           \draw[style=semithick, dashed, -latex] (p62.south)--(p61.north);
           \draw[style=semithick, dashed, -latex] (p63.south)--(p62.north);
           \draw[style=semithick, dashed, -latex] (p64.south)--(p63.north);
           \draw[style=semithick, -latex] (p61.west)--(p51.east);
          \draw[style=semithick, -latex] (p62.west)--(p52.east);
          \draw[style=semithick, -latex] (p63.west)--(p53.east);
          \draw[style=semithick, -latex] (p64.west)--(p54.east);

           \node at (21,0) {.}; \node at (21.5,0) {.}; \node at (22,0) {.};
          \node at (21,4) {.}; \node at (21.5,4) {.}; \node at (22,4) {.};
          \node at (21,8) {.}; \node at (21.5,8) {.}; \node at (22,8) {.};
          \node at (21,12) {.}; \node at (21.5,12) {.}; \node at (22,12) {.};
           \node at (0,13) {.}; \node at (0,13.5) {.}; \node at (0,14) {.};
           \node at (4,13) {.}; \node at (4,13.5) {.}; \node at (4,14) {.};
           \node at (8,13) {.}; \node at (8,13.5) {.}; \node at (8,14) {.};
           \node at (12,13) {.}; \node at (12,13.5) {.}; \node at (12,14) {.};
          \node at (16,13) {.}; \node at (16,13.5) {.}; \node at (16,14) {.};
          \node at (20,13) {.}; \node at (20,13.5) {.}; \node at (20,14) {.};

\end{tikzpicture}
\]

\noindent Then $\Lambda$ is cofinal: For example for $v,w$ as shown, $N=(1,0)$ will suffice. However $E_\Lambda$ is not cofinal: For example for $v,w$ as shown, for any $n \ge 0$ the vertex which is the source of the vertical path of length $n$ with range $w$ does not connect to $v$.
\end{rmk}

\noindent The following proposition establishes a link between cofinality and strongly connectivity for a row-finite $k$-graph.

\begin{prop} \label{prop:cofinal-SC}
Suppose $\Lambda$ is a row-finite $k$-graph with no sources.
\begin{enumerate}
\item If $\Lambda$ is strongly connected then $\Lambda$ is cofinal.
\item If $\Lambda$ is cofinal, has no sinks and $\Lambda^0$ finite then $\Lambda$ is strongly connected.
\end{enumerate}
\end{prop}

\begin{proof}
Suppose $\Lambda$ is strongly connected. Fix $v,w \in \Lambda^0$ then for $N=e_1$ we have $v \Lambda s ( \alpha ) \neq \emptyset$ for all $\alpha \in w \Lambda^N$ since $\Lambda$ is strongly connected, and so $\Lambda$ is cofinal.

Suppose $\Lambda$ is cofinal. Fix $u,v\in\Lambda^{0}$. Then by Lemma \ref{lemma:cofinal-loops}, there exists $w\in\Lambda^{0}$ and $\alpha\in w\Lambda w$ such that $d(\alpha)>0$ and $w\Lambda v\neq\emptyset$. Let $\alpha^{\prime}\in w\Lambda v$. Given $u,w\in\Lambda^{0}$, since $\Lambda$ is cofinal and has no sources, by Lemma \ref{lem:cofinalconn}(ii) there exists $N\in\N^{k}$ such that for all $n\geq N$ and all $\alpha^{\prime\prime}\in w\Lambda^{n}$, there exists $\beta\in u\Lambda s(\alpha^{\prime\prime})$. Since $d ( \alpha ) > 0$ we may choose $t\in\N$ such that $td(\alpha)>N$. Then $\alpha^t \in w\Lambda^{n}$ where $n>N$, and so by cofinality of $\Lambda$ exists $\beta\in u\Lambda s(\alpha^t)=u\Lambda w$. Hence $\beta\alpha\alpha^{\prime}\in u\Lambda v$ with $d(\beta\alpha\alpha^{\prime})>d(\alpha)>0$ and so $\Lambda$ is strongly connected.
\end{proof}

\begin{example}
The condition that $\Lambda^0$ is finite in Proposition~\ref{prop:cofinal-SC}(2) is essential: For instance $\Delta_k$ is cofinal by Lemma~\ref{lem:skelcofinal} since its skeleton is cofinal; however it is not strongly connected by Lemma~\ref{lem:skelcofinal} since its skeleton is not strongly connected.
\end{example}

Since the map $\pi : \Lambda \times_\eta S \to \Lambda$ has unique $r$-path lifting, we wish to know if we can deduce the cofinality of $\Lambda \times_\eta S$ from that of $\Lambda$. By Theorem~\ref{thm:lifts} the image of a cofinal $k$-graph under a map with $r$-path lifting is cofinal, however Example~\ref{ex:spgnotconn} shows that the converse is not true. For a cofinal $k$-graph $\Lambda$, we must then seek additional conditions on the functor $\eta$ which guarantees that $\Lambda \times_\eta S$ is cofinal. In Definition \ref{def:etacofinal} we introduce the notion of $( \Lambda , S, \eta)$ cofinality to address this problem.

\begin{thm} \label{thm:lifts}
Suppose $\Lambda, \Gamma$ be row-finite $k$-graphs and $p : \Lambda \to \Gamma$ have $r$-path lifting. If $\Lambda$ is cofinal then $\Gamma$ is cofinal.
\end{thm}

\begin{proof}
Suppose that $\Lambda$ is cofinal. Fix $v,w \in \Gamma^0$. Let $v',w' \in \Lambda^0$ be such that $p (v' ) = v$ and $p ( w' ) = w$. Since $\Lambda$ is cofinal there is an $N$ such that for all $\alpha' \in w' \Lambda^N$ there is $\beta' \in v' \Lambda s( \alpha' )$. Then for $\alpha \in v \Gamma^N$ there is $\alpha' \in v' \Lambda^N$ with $p ( \alpha' )= \alpha$. By hypothesis there is $\beta' \in v' \Lambda s ( \alpha' )$, and so $\beta = p ( \beta' )$ is such that $s ( \beta ) = s ( \alpha )$ and $r ( \beta ) = v$ which implies that $v \Lambda s ( \alpha ) \neq \emptyset$ as required.
\end{proof}

\begin{cor} \label{cor:spg2}
Let $\Lambda$ be a row-finite $k$-graph with no sources, $\eta : \Lambda \to S$ a functor where $S$ is a semigroup and $\Lambda \times_\eta S$ the associated skew product graph. If $\Lambda \times_\eta S$ is cofinal then $\Lambda$ is cofinal.
\end{cor}

\noindent
The converse of Theorem~\ref{thm:lifts} is false:

\begin{example} \label{ex:spgnotconn}
Consider the following $2$-graph $\Lambda$ with $1$-skeleton
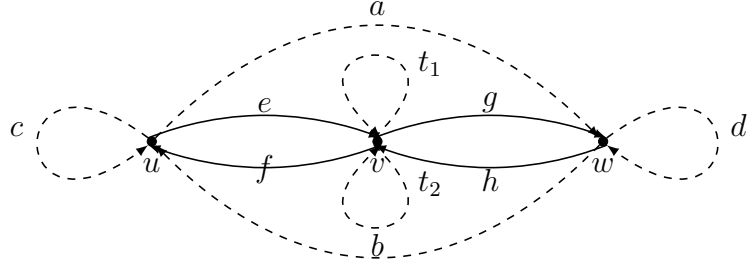
\begin{equation*}
\begin{tikzpicture}
\def\vertex(#1) at (#2,#3){
    \node[inner sep=0pt, circle, fill=black] (#1) at (#2,#3)
    [draw] {.};
  }

  \vertex(vertA) at (-3,0) \node[inner sep=2pt, anchor = north] at (vertA.south) at (-3,-0.15) {$u$};
  \vertex(vertB) at (0,0) \node[inner sep=2pt, anchor = north] at (vertB.south) at (0,-0.15) {$v$};
  \vertex(vertC) at (3,0) \node[inner sep=2pt, anchor = north] at (vertC.south) at (3,-0.15) {$w$};

  \draw[style=semithick,-latex] (vertA.north west) parabola[parabola height=0.3cm] (vertB.north east);
  \node[inner sep=1pt, anchor = south] at (-1.5,0.35) {$e$};
  \draw[style=semithick,-latex] (vertB.south east) parabola[parabola height=-0.3cm] (vertA.south west);
  \node[inner sep=1pt, anchor = south] at (-1.5,-0.6) {$f$};
  \draw[style=semithick,-latex] (vertB.north west) parabola[parabola height=0.3cm] (vertC.north east);
  \node[inner sep=1pt, anchor = south] at (1.5,0.35) {$g$};
  \draw[style=semithick,-latex] (vertC.south east) parabola[parabola height=-0.3cm] (vertB.south west);
  \node[inner sep=1pt, anchor = south] at (1.5,-0.7) {$h$};
  \draw[style=semithick, style=dashed, -latex] (vertA.north east) parabola[parabola height=1.5cm] (vertC.north west);
  \node[inner sep=2pt, anchor = south] at (0,1.6) {$a$};
  \draw[style=semithick, style=dashed, -latex] (vertC.south west) parabola[parabola height=-1.5cm] (vertA.south east);
  \node[inner sep=2pt, anchor = south] at (0,-1.6) {$b$};
  \draw[style=semithick, style=dashed, -latex] (vertA.north west) .. controls (-5,1.5) and (-5,-1.5) .. (vertA.south west); \node[inner sep=2pt, anchor = south] at (-4.8,0) {$c$};
  \draw[style=semithick, style=dashed, -latex] (vertC.north east) .. controls (5,1.5) and (5,-1.5) .. (vertC.south east); \node[inner sep=2pt, anchor = south] at (4.8,0) {$d$};
  \draw[style=semithick, style=dashed, -latex] (vertB.north west) .. controls (1.5,1.5) and (-1.5,1.5) .. (vertB.north east);
  \node[inner sep=2pt, anchor = south] at (0.7,0.8) {$t_1$};
  \draw[style=semithick, style=dashed, -latex] (vertB.south west) .. controls (1.5,-1.5) and (-1.5,-1.5) .. (vertB.south east);
  \node[inner sep=2pt, anchor = south] at (0.7,-0.8) {$t_2$};
\end{tikzpicture}
\end{equation*}

\noindent and factorisation rules: $ec=t_1e$ and $ha=t_2e$ for paths from $u$ to $v$; $cf=ft_1$ and $bg=ft_2$ for paths from $v$ to $u$. Also $hd=t_1h$ and $eb=t_2h$ for paths from $w$ to $v$; $dg=gt_1$ and $af=gt_2$ for paths from $v$ to $w$. By Lemma~\ref{lem:skelcofinal} $\Lambda$ is strongly connected as its skeleton is strongly connected. Note there are no paths of degree $e_1+e_2$ from a vertex to itself.

Since $M_{1}=\left(\begin{smallmatrix} 0 & 1 & 0 \\ 1 & 0 & 1 \\ 0 & 1 & 0 \\\end{smallmatrix}\right)$ and $M_{2}=\left(\begin{smallmatrix} 1 & 0 & 1 \\ 0 & 2 & 0 \\ 1 & 0 & 1 \\\end{smallmatrix}\right)$, we calculate that $M^{(2j_1,j_2)} = 2^{j_1+j_2-1} M_2$ and $M^{(2j_1+1,j_2)} = 2^{j_1+j_2+1} M_1$. Hence $M^{(2j-1,2j-1)}=\left(\begin{smallmatrix} 0 & 4j & 0 \\ 4j & 0 & 4j \\ 0 & 4j & 0 \end{smallmatrix}\right)$ and
$M^{(2j,2j)}=\left(\begin{smallmatrix} 4j & 0 & 4j \\ 0 & 8j & 0 \\ 4j & 0 & 4j \end{smallmatrix}\right)$. In particular by Lemma~\ref{lem:matrixcond}~(b) $\Lambda$ is not primitive, even though it is strongly connected.

We claim that the skew product graph $\Lambda \times_d \ZZ^2$ is not cofinal. Consider $v_1 = ( v , (m ,n) )$ and $v_2 = ( v , (m+1,n))$ in $(\Lambda\times_d\ZZ^2)^{0}$. We claim that for all $N \in \NN^2$, for all $\alpha \in v_1 (\Lambda \times_d \ZZ^2)^N$, we have $v_2 (\Lambda \times_d \ZZ^2) s ( \alpha ) \neq \emptyset$. Let $N = (N_1, N_2)$. Suppose $N_1$ is even. Then for all $\alpha \in v_1 (\Lambda \times_d \ZZ^2)^N$, $s ( \alpha ) = ( v , (m + N_1 , n + N_2) )$. In order for this vertex to connect to $( v , ( m+1 , n ) )$, we have $M^{(N_1 - 1 , N_2 )} ( v,v ) \neq 0$. But $N_1 - 1$ is odd, and this matrix entry is zero. If $N_1$ is odd, then $s(\alpha)=(u,(m+N_1,n+N_2))$ or $s(\alpha)=(w,(m+N_1,n+N_2))$. In order for either of these vertices to connect to $(v,(m+1,n))$, we must have $M^{(N_1-1,N_2)}(u,v)\neq0$, or $M^{(N_1-1,N_2)}(w,v)\neq0$. But $N_1-1$ is even, and so both of these matrix entries are zero. Hence $\Lambda\times_d\ZZ^2$ is not cofinal, even though $\Lambda$ is cofinal.
\end{example}

\noindent To establish a sufficient condition for $\Lambda\times_{\eta}S$ to be cofinal, we need $\Lambda$ to be cofinal and an additional condition on $\eta$.

\begin{dfn} \label{def:etacofinal}
Let $\Lambda$ be a $k$-graph with no sources and $\eta : \Lambda \to S$ a functor, where $S$ is a semigroup. The system $( \Lambda , S , \eta )$ is \textit{cofinal} if for all $v,w\in\Lambda^{0}$, $a,b\in S$, there exists $N \in\N^{k}$ such that for all $\alpha\in w\Lambda^{N}$, there exists $\beta\in v\Lambda s(\alpha)$ such that $a\eta(\beta)=b\eta(\alpha)$.
\end{dfn}

\begin{prop} \label{prop:eta-cofinal}
Let $\Lambda$ be a $k$-graph with no sources and $\eta : \Lambda \to S$ a functor, where $S$ is a semigroup and $\Lambda \times_\eta S$ the associated skew product graph. Then the system $(\Lambda,S,\eta)$ is cofinal if and only if $\Lambda\times_{\eta} S$ is cofinal.
\end{prop}

\begin{proof}
Suppose $\Lambda\times_{\eta}S$ is cofinal. Fix $a,b \in S$ and $v,w \in \Lambda^0$. By hypothesis there is $N \in\N^{k}$ such that $(v,a)(\Lambda\times_{\eta}S)s(\alpha,b)$ is non-empty for every $(\alpha,b)\in(w,b)(\Lambda\times_{\eta}S)^{N}$. In particular for all $\alpha \in w \Lambda^N$ there exists $\beta \in w \Lambda^N$ such that $a\eta(\beta)=b\eta(\alpha)$, and so $( \Lambda , S , \eta )$ is cofinal.

Now suppose $(\Lambda,S,\eta)$ is cofinal. Fix $(v,a) , (w,b ) \in ( \Lambda \times_\eta S )^0$. By hypothesis there exists $N\in\N^{k}$ such that for all $\alpha\in w\Lambda^{N}$, there exists $\beta\in v\Lambda s(\alpha)$ with $a\eta(\beta)=b\eta(\alpha)$. In particular for all $( \alpha , b ) \in (w,b) ( \Lambda \times_\eta S )^N$ there is $( \beta , a ) \in (v,a) \Lambda s ( \alpha , b )$, and so $\Lambda\times_{\eta}S$ is cofinal.
\end{proof}


\begin{thm} \label{thm:cofinal}
Let $\Lambda$ be an aperiodic row-finite $k$-graph with no sources, $\eta : \Lambda \to S$ a functor, where $S$ is a semigroup and $\Lambda \times_\eta S$ the associated skew product graph. Then $C^* ( \Lambda\times_{\eta}S )$ is simple if and only if the system $(\Lambda , S , \eta )$ is cofinal.
\end{thm}

\begin{proof}
Suppose that the system $(\Lambda , S , \eta )$ is cofinal. Then by Proposition \ref{prop:eta-cofinal}, $\Lambda\times_{\eta}S$ is cofinal. By Corollary~\ref{cor:spg1}, $\Lambda \times_\eta S$ is aperiodic and so by \cite[Theorem 3.1]{RobS}, $C^* ( \Lambda\times_{\eta}S )$ is simple.

Now suppose that $C^* ( \Lambda \times_\eta S )$ is simple. Then by \cite[Theorem 3.1]{RobS}, $\Lambda \times_\eta S$ is cofinal. By Proposition~\ref{prop:eta-cofinal} this implies that $(\Lambda , S , \eta )$ is cofinal.
\end{proof}

The condition of $( \Lambda , S, \eta )$ cofinality is difficult to check in practice. For $1$-graphs it was shown in \cite[Proposition 5.13]{PRho} that $\Lambda \times_d \ZZ^k$ is cofinal if $\Lambda$ is primitive\footnote{Actually strongly connected with period $1$ which is equivalent to primitive}. We seek an equivalent condition for $k$-graphs which guarantees $( \Lambda , S, \eta )$ cofinality.

\section{Primitivity and left-reversible semigroups} \label{sec:primitivity}

A semigroup $S$ is said to be \emph{left-reversible} if for all $s , t \in S$ we have $sS \cap tS \neq \emptyset$. It is more common to work with right-reversible semigroups, which are then called Ore semigroups (see \cite{MPR}). In analogy with the results of Dubriel it can be shown that a left-reversible semigroup has an enveloping group $\Gamma$ such that $\Gamma = S S^{-1}$.

In equation \eqref{eq:rsdef} we see that functor $\eta : \Lambda \to S$ multiplies on the right in the semigroup coordinate in the definition of the source map in a skew product graph $\Lambda \times_\eta S$. This forces us to consider left-reversible semigroups here. In order to avoid confusion we have decided not to call them Ore.

\begin{examples}
\begin{enumerate}[(i)]
\item Any abelian semigroup is automatically right- and left-reversible. Moreover, any group is a both a right- and left-reversible semigroup.
\item Let $\mathbb{N}$ denote the semigroup of natural numbers under addition and $\mathbb{N}^\times$ denote the semigroup of nonzero natural numbers under multiplication. Let $S = \mathbb{N} \times \mathbb{N}^\times$ be gifted with the associative binary operation $\star$ given by
\[
( m_1 , n_1 ) \star ( m_2 , n_2 ) = ( m_1 n_2 + m_2 , n_1 n_2 ) ,
\]
\noindent then one checks that $S$ is a nonabelian left-reversible semigroup. It is not right-reversible; for example, $S (m,n) \cap S(p,q) = \emptyset$ when $n=q=0$ and $m\neq p$.
\item The free semigroup $\mathbb{F}_n^+$ on $n \geq 2$ generators is not an left-reversible semigroup since for all $s,t \in \mathbb{F}_n^+$ with $s \neq t$ we have $s \mathbb{F}_{n}^{+} \cap t \mathbb{F}_{n}^{+} = \emptyset$ as there is no cancellation, and so not only the left-reversibility but also the right-reversibility conditions cannot be satisfied.
\end{enumerate}
\end{examples}

A \textit{preorder} is a reflexive, transitive relation $\leq$ on a set $X$. A preordered set $(X,\leq)$ is \textit{directed} if the following condition holds: for every $x,y\in X$, there exists $z\in X$ such that $x\leq z$ and $y\leq z$. A subset $Y$ of $X$ is \textit{cofinal} if for each $x\in X$ there exists $y\in Y$ such that $x\leq y$. We say that sets $X\leq Y$ if $x\leq y$ for all $x \in X$ and for all $y \in Y$. We say that $t\in S$ is \textit{strictly positive} if $\{t^n : n \ge 0 \}$ is a cofinal set in $S$.

The following result appears as \cite[Lemma 2.2]{PRY} for right-reversible semigroups.

\begin{lem}
\label{lemma:preceqdirect}
Let $S$ be a left-reversible semigroup with enveloping group $\Gamma$, and define $\geq_{l}$ on $\Gamma$ by $h\geq_{l}g$ if and only if $g^{-1}h \in S$. Then $\geq_{l}$ is a left-invariant preorder that directs $\Gamma$, and for any $t\in S$, $tS$ is cofinal in $S$.
\end{lem}

\noindent
Our first attempt at a condition on $\eta$ which guarantees cofinality of $( \Lambda, S , \eta )$
is one which ensures that $\eta$ takes arbitrarily
large values on paths which terminate a given vertex.
%




\begin{dfn}
Let $\Lambda$ be a $k$-graph with no sources and $\eta : \Lambda \to S$ be a functor where $S$ is a left-reversible semigroup.
We will say that $\eta$ is \textit{upper dense} if for all $w \in \Lambda^0$ and $a,b \in S$ there exists $N \in \mathbb{N}^k$ such that $b \eta ( w \Lambda^N ) \geq_l a$.
\end{dfn}

\begin{lem} \label{lem:dud}
Let $(\Lambda,d)$ be a row-finite $k$-graph with no sources then $d$ is upper dense for $\Lambda$.
\end{lem}

\begin{proof}
Since $\Lambda$ has no sources it is immediate that $w \Lambda^N \neq \emptyset$ for all $w \in \Lambda^0$ and $N \in \NN^k$.
For any $b,a \in \NN^k$ we have $b + d ( w \Lambda^N ) = b + N \geq a$ provided $N \geq a$.
\end{proof}

\begin{examples} \label{ex:S-ud}
\begin{enumerate}[(i)]
\item Let $B_2$ be the $1$-graph which is the path category of the directed graph with a single vertex $v$ and two edges $e,f$. Define a functor $\eta : B_2 \to \NN$ by $\eta (e)=1$ and $\eta (f)=0$. We may form the skew product $B_2 \times_\eta \NN$ with $1$-skeleton:
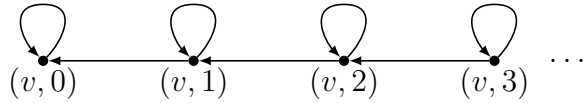
\[
\begin{tikzpicture}
  \def\vertex(#1) at (#2,#3){
    \node[inner sep=0pt, circle, fill=black] (#1) at (#2,#3)
    [draw] {.};
  }
  \vertex(11) at (0, 1)
  \vertex(21) at (2, 1)
\vertex(31) at (4, 1)
\vertex(41) at (6, 1)
 \node at (7, 1) {$\dots$};
\node at (0,0.7) {$(v,0)$};
\node at (2,0.7) {$(v,1)$};
\node at (4,0.7) {$(v,2)$};
\node at (6,0.7) {$(v,3)$};

\draw[style=semithick, -latex] (21.west)--(11.east);
\draw[style=semithick, -latex] (31.west)--(21.east);
 \draw[style=semithick, -latex] (41.west)--(31.east);

\draw[style=semithick, -latex] (11.north east)
    .. controls (0.25,1.25) and (0.5,1.75) ..
    (0,1.75)
    .. controls (-0.5,1.75) and (-0.25,1.25) ..
    (11.north west);
 \draw[style=semithick, -latex] (21.north east)
    .. controls (2.25,1.25) and (2.5,1.75) ..
    (2,1.75)
    .. controls (1.5,1.75) and (1.75,1.25) ..
    (21.north west);
  \draw[style=semithick, -latex] (31.north east)
    .. controls (4.25,1.25) and (4.5,1.75) ..
    (4,1.75)
    .. controls (3.5,1.75) and (3.75,1.25) ..
    (31.north west);
  \draw[style=semithick, -latex] (41.north east)
    .. controls (6.25,1.25) and (6.5,1.75) ..
    (6,1.75)
    .. controls (5.5,1.75) and (5.75,1.25) ..
    (41.north west);
\end{tikzpicture}
\]

\noindent Fix $a,b \in \NN$, then since $n \in \eta ( v B_2^n )$ for all $n \in \NN$ it follows that if we choose $N=a$,
then $b + \eta ( v B_2^N ) \ge a$ and so $\eta$ is upper dense. However $( B_2 , \NN , \eta )$ is not cofinal: Choose $a=1$, $b=0$, then for all $N \ge 0$ there is $f^N \in v B_2^N$ is such that
\[
b+ \eta ( f^N ) = 0 \neq 1 + \eta ( \beta ) \text{ for all } \beta \in B_2 v.
\]

\item Define a functor $\eta$ from $T_2$ to $\NN^2$ such that $\eta(f_1)=(2,0)$, and $\eta(f_2)=(0,1)$. We may form the skew product $T_2\times_{\eta} \NN^2$ with the following $1$-skeleton:
  
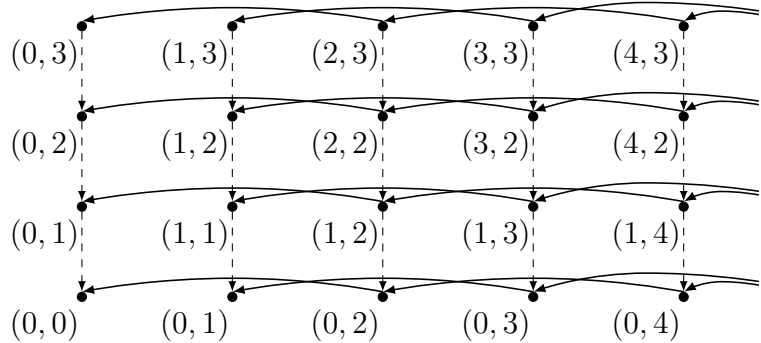
\[\begin{tikzpicture}[yscale=0.6]
  \def\vertex(#1) at (#2,#3){
    \node[inner sep=0pt, circle, fill=black] (#1) at (#2,#3)
    [draw] {.};}
  \def\invVertex(#1) at (#2,#3){
    \node[inner sep=0pt] (#1) at (#2,#3) {};}
  \vertex(vert01) at (-4,0) \node[inner sep=2pt, anchor = north] at (vertA.south) at (-4.5,-0.15) {$(0,0)$};
  \vertex(vert11) at (-2,0) \node[inner sep=2pt, anchor = north] at (vertA.south) at (-2.5,-0.15) {$(0,1)$};
  \vertex(vert21) at (0,0) \node[inner sep=2pt, anchor = north] at (vertA.south) at (-0.5,-0.15) {$(0,2)$};
  \vertex(vert31) at (2,0) \node[inner sep=2pt, anchor = north] at (vertA.south) at (1.5,-0.15) {$(0,3)$};
  \vertex(vert41) at (4,0) \node[inner sep=2pt, anchor = north] at (vertA.south) at (3.5,-0.15) {$(0,4)$};

  \vertex(vert02) at (-4,2) \node[inner sep=2pt, anchor = north] at (vertA.south) at (-4.5,1.85) {$(0,1)$};
  \vertex(vert12) at (-2,2) \node[inner sep=2pt, anchor = north] at (vertA.south) at (-2.5,1.85) {$(1,1)$};
  \vertex(vert22) at (0,2) \node[inner sep=2pt, anchor = north] at (vertA.south) at (-0.5,1.85) {$(1,2)$};
  \vertex(vert32) at (2,2) \node[inner sep=2pt, anchor = north] at (vertA.south) at (1.5,1.85) {$(1,3)$};
  \vertex(vert42) at (4,2) \node[inner sep=2pt, anchor = north] at (vertA.south) at (3.5,1.85) {$(1,4)$};

  \vertex(vert04) at (-4,4) \node[inner sep=2pt, anchor = north] at (vertA.south) at (-4.5,3.85) {$(0,2)$};
  \vertex(vert14) at (-2,4) \node[inner sep=2pt, anchor = north] at (vertA.south) at (-2.5,3.85) {$(1,2)$};
  \vertex(vert24) at (0,4) \node[inner sep=2pt, anchor = north] at (vertA.south) at (-0.5,3.85) {$(2,2)$};
  \vertex(vert34) at (2,4) \node[inner sep=2pt, anchor = north] at (vertA.south) at (1.5,3.85) {$(3,2)$};
  \vertex(vert44) at (4,4) \node[inner sep=2pt, anchor = north] at (vertA.south) at (3.5,3.85) {$(4,2)$};

  \vertex(vert08) at (-4,6) \node[inner sep=2pt, anchor = north] at (vertA.south) at (-4.5,5.85) {$(0,3)$};
  \vertex(vert18) at (-2,6) \node[inner sep=2pt, anchor = north] at (vertA.south) at (-2.5,5.85) {$(1,3)$};
  \vertex(vert28) at (0,6) \node[inner sep=2pt, anchor = north] at (vertA.south) at (-0.5,5.85) {$(2,3)$};
  \vertex(vert38) at (2,6) \node[inner sep=2pt, anchor = north] at (vertA.south) at (1.5,5.85) {$(3,3)$};
  \vertex(vert48) at (4,6) \node[inner sep=2pt, anchor = north] at (vertA.south) at (3.5,5.85) {$(4,3)$};
  \invVertex(vert51) at (5,0.3);
  \invVertex(vert52) at (5,2.3);
  \invVertex(vert54) at (5,4.3);
  \invVertex(vert58) at (5,6.3);
  \draw[style=dashed, -latex] (vert08.south)--(vert04.north);
  \draw[style=dashed, -latex] (vert04.south)--(vert02.north);
  \draw[style=dashed, -latex] (vert02.south)--(vert01.north);
  \draw[style=dashed, -latex] (vert18.south)--(vert14.north);
  \draw[style=dashed, -latex] (vert14.south)--(vert12.north);
  \draw[style=dashed, -latex] (vert12.south)--(vert11.north);
  \draw[style=dashed, -latex] (vert28.south)--(vert24.north);
  \draw[style=dashed, -latex] (vert24.south)--(vert22.north);
  \draw[style=dashed, -latex] (vert22.south)--(vert21.north);
  \draw[style=dashed, -latex] (vert38.south)--(vert34.north);
  \draw[style=dashed, -latex] (vert34.south)--(vert32.north);
  \draw[style=dashed, -latex] (vert32.south)--(vert31.north);
  \draw[style=dashed, -latex] (vert48.south)--(vert44.north);
  \draw[style=dashed, -latex] (vert44.south)--(vert42.north);
  \draw[style=dashed, -latex] (vert42.south)--(vert41.north);

  \draw[style=semithick, -latex] (vert41.north) parabola[parabola height=0.3cm] (vert21.north);
  \draw[style=semithick, -latex] (vert21.north) parabola[parabola height=0.3cm] (vert01.north);
  \draw[style=semithick, -latex] (vert31.north) parabola[parabola height=0.3cm] (vert11.north);
  \draw[style=semithick, -latex] (vert51.north) parabola[parabola height=0.3cm] (vert31.north);
  \draw[style=semithick, -latex] (vert51.south) parabola[parabola height=0.15cm] (vert41.north);

  \draw[style=semithick, -latex] (vert22.north) parabola[parabola height=0.3cm] (vert02.north);
  \draw[style=semithick, -latex] (vert32.north) parabola[parabola height=0.3cm] (vert12.north);
  \draw[style=semithick, -latex] (vert42.north) parabola[parabola height=0.3cm] (vert22.north);
  \draw[style=semithick, -latex] (vert52.north) parabola[parabola height=0.3cm] (vert32.north);
  \draw[style=semithick, -latex] (vert52.south) parabola[parabola height=0.15cm] (vert42.north);

  \draw[style=semithick, -latex] (vert24.north) parabola[parabola height=0.3cm] (vert04.north);
  \draw[style=semithick, -latex] (vert34.north) parabola[parabola height=0.3cm] (vert14.north);
  \draw[style=semithick, -latex] (vert44.north) parabola[parabola height=0.3cm] (vert24.north);
  \draw[style=semithick, -latex] (vert54.north) parabola[parabola height=0.3cm] (vert34.north);
  \draw[style=semithick, -latex] (vert54.south) parabola[parabola height=0.15cm] (vert44.north);

  \draw[style=semithick, -latex] (vert28.north) parabola[parabola height=0.3cm] (vert08.north);
  \draw[style=semithick, -latex] (vert38.north) parabola[parabola height=0.3cm] (vert18.north);
  \draw[style=semithick, -latex] (vert48.north) parabola[parabola height=0.3cm] (vert28.north);
  \draw[style=semithick, -latex] (vert58.north) parabola[parabola height=0.3cm] (vert38.north);
  \draw[style=semithick, -latex] (vert58.south) parabola[parabola height=0.15cm] (vert48.north);
  \end{tikzpicture}\]

\noindent We claim that the functor $\eta$ is not upper dense:
Fix $b=(b_1,b_2)$ and $a=(a_1,a_2)$ in $\NN^2$. Let $N_1$ be such that $b_1 +2N_1 \ge a_1$ and $N_2$ be such that $b_2 +N_2 \ge a_2$ then $b \eta ( v T_2^N ) \ge_l u$ where $N=(N_1,N_2)$. Moreover $( T_2 , \NN^2 , \eta )$ is not cofinal: Let $b=(0,0)$ and $a=(1,0)$ then since $\eta (f_1^{N_1} f_2^{N_2} ) = (2N_1,N_2)$ it follows that there cannot be $N=(N_1,N_2)\in \NN^2$ such that for $\alpha \in v T_2^N$ there is $\beta \in v T_2 v$ with $b \eta (\alpha ) = a \eta ( \beta )$.
\item Taking $T_2$ again, we define a functor $\eta:T_2\rightarrow\NN^2$ by $\eta(f_1)=(1,0)$ and $\eta(f_2)=(1,1)$. The skew product graph has $1$-skeleton:
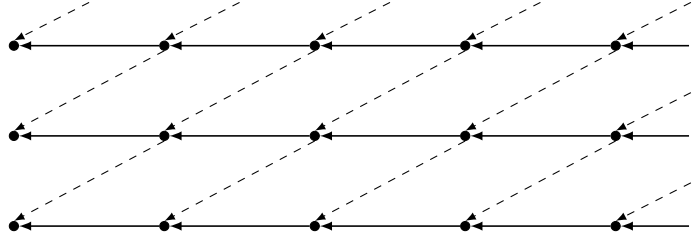
\[\begin{tikzpicture}[yscale=0.6]
  \def\vertex(#1) at (#2,#3){ \node[inner sep=0pt, circle, fill=black] (#1) at (#2,#3) [draw] {.}; }
  \def\invVertex(#1) at (#2,#3){ \node[inner sep=0pt] (#1) at (#2,#3) {}; }
  \vertex(vert00) at (0,0);
  \vertex(vert10) at (2,0);
  \vertex(vert20) at (4,0);
  \vertex(vert30) at (6,0);
  \vertex(vert40) at (8,0);
  \vertex(vert01) at (0,2);
  \vertex(vert11) at (2,2);
  \vertex(vert21) at (4,2);
  \vertex(vert31) at (6,2);
  \vertex(vert41) at (8,2);
  \vertex(vert02) at (0,4);
  \vertex(vert12) at (2,4);
  \vertex(vert22) at (4,4);
  \vertex(vert32) at (6,4);
  \vertex(vert42) at (8,4);
  \invVertex(vert50) at (9,0);
  \invVertex(vert50a) at (9,1);
  \invVertex(vert51) at (9,2);
  \invVertex(vert51a) at (9,3);
  \invVertex(vert52) at (9,4);
  \invVertex(vert0a3) at (1,5);
  \invVertex(vert1a3) at (3,5);
  \invVertex(vert2a3) at (5,5);
  \invVertex(vert3a3) at (7,5);
  \invVertex(vert4a3) at (9,5);
  \draw[style=semithick, -latex] (vert10.west)--(vert00.east);
  \draw[style=semithick, -latex] (vert20.west)--(vert10.east);
  \draw[style=semithick, -latex] (vert30.west)--(vert20.east);
  \draw[style=semithick, -latex] (vert40.west)--(vert30.east);
  \draw[style=semithick, -latex] (vert11.west)--(vert01.east);
  \draw[style=semithick, -latex] (vert21.west)--(vert11.east);
  \draw[style=semithick, -latex] (vert31.west)--(vert21.east);
  \draw[style=semithick, -latex] (vert41.west)--(vert31.east);
  \draw[style=semithick, -latex] (vert12.west)--(vert02.east);
  \draw[style=semithick, -latex] (vert22.west)--(vert12.east);
  \draw[style=semithick, -latex] (vert32.west)--(vert22.east);
  \draw[style=semithick, -latex] (vert42.west)--(vert32.east);
  \draw[style=dashed, -latex] (vert12.south)--(vert01.north);
  \draw[style=dashed, -latex] (vert22.south)--(vert11.north);
  \draw[style=dashed, -latex] (vert32.south)--(vert21.north);
  \draw[style=dashed, -latex] (vert42.south)--(vert31.north);
  \draw[style=dashed, -latex] (vert11.south)--(vert00.north);
  \draw[style=dashed, -latex] (vert21.south)--(vert10.north);
  \draw[style=dashed, -latex] (vert31.south)--(vert20.north);
  \draw[style=dashed, -latex] (vert41.south)--(vert30.north);
  \draw[style=semithick, -latex] (vert50.west)--(vert40.east);
  \draw[style=semithick, -latex] (vert51.west)--(vert41.east);
  \draw[style=semithick, -latex] (vert52.west)--(vert42.east);
  \draw[style=dashed, -latex] (vert50a.south)--(vert40.north);
  \draw[style=dashed, -latex] (vert51a.south)--(vert41.north);
  \draw[style=dashed, -latex] (vert0a3.south)--(vert02.north);
  \draw[style=dashed, -latex] (vert1a3.south)--(vert12.north);
  \draw[style=dashed, -latex] (vert2a3.south)--(vert22.north);
  \draw[style=dashed, -latex] (vert3a3.south)--(vert32.north);
  \draw[style=dashed, -latex] (vert4a3.south)--(vert42.north);
\end{tikzpicture}\]

\noindent We claim that $\eta$ is upper dense: Fix $b=(b_2,b_2)$ and $a=(a_1,a_2)$ in $\NN^2$ then there is $N_1$ such that $b_1 + N_1 \ge a_1$ and $N_2$ such that $b_2 +N_1+N_2 \ge a_2$. Then with $N=(N_1,N_2)$ for all $\alpha \in v T_2^N$ we have $b \eta ( \alpha ) \ge_l a$. In this case $( T_2 , \NN^2 , \eta )$ is cofinal: Fix $b=(b_1, b_2)$ and $a=(a_1,a_2)$ in $\NN^2$. Then there is $N_1$ such that $b_1+N_1=a_1+m_1$ for some $m_1 \in \NN$ and $N_2$ such that $b_2+N_1+N_2=a_2+m_2$ for some $m_2 \in \NN$. Hence for all $\alpha \in v T_2^N$ where $N=(N_1,N_2)$ there is $\beta = (f_1^{m_1},f_2^{m_2}) \in v T_2 v$ such that $b \eta ( \alpha )= a \eta ( \beta )$.
\end{enumerate}
\end{examples}

It is clear from these last two examples that $\eta$ being upper dense is not sufficient to guarantee cofinality of $( \Lambda , S , \eta )$. The following definition allows for the interaction of the values of $\eta$ at different vertices of $\Lambda$ and the following result gives us the required extra condition.

\begin{dfn}
Let $\Lambda$ be a $k$-graph and $\eta : \Lambda \to S$ be a functor where $S$ is a left-reversible semigroup. We say that $\eta$ is \textit{$S$-primitive} for $\Lambda$ if there is a strictly positive $t \in S$ such that for all $v , w \in \Lambda^0$ we have $v \eta^{-1} ( s ) w \neq \emptyset$ for all $s \in S$ such that $s \geq_{l} t$.
\end{dfn}

\begin{rmks}
\begin{enumerate}[(i)]
\item The condition that $t$ is strictly positive in the above definition guarantees that $\eta ( v \Lambda w )$ is cofinal in $S$ for all $v,w \in \Lambda^0$.
\item If $\eta : \Lambda \to S$ is $S$-primitive for $\Lambda$ where $S$ is a left-reversible semigroup, then if we extend $\eta$ to $\Gamma = S S^{-1}$ then $\eta$ is $\Gamma$-primitive for $\Lambda$.
\end{enumerate}
\end{rmks}

\begin{examples} \label{ex:S-prim}
\begin{enumerate}[(i)]
\item Let $\Lambda$ be a $k$-graph. Then the degree functor $d : \Lambda \to \NN^k$ is $\NN^{k}$--primitive for $\Lambda$ if and only if $\Lambda$ is primitive as defined in Section \ref{ss:conn}. For this reason we will say that $\Lambda$ is primitive if $d$ is $\NN^k$ primitive for $\Lambda$.
\item As in Examples~\ref{ex:S-ud} (i) let $\eta : B_2 \to \NN$ be defined by $\eta (e)=1$, $\eta (f)=0$. Then the functor $\eta$ is $\NN$-primitive since $\eta^{-1} (n)$ is nonempty for all $n \in \NN$. Hence $\NN$-primitivity does not, by itself, guarantee cofinality.
\item As in Examples~\ref{ex:S-ud} (ii) let $\eta$ be the functor from $T_2$ to $\NN^2$ such that $\eta(f_1)=(2,0)$, and $\eta(f_2)=(0,1)$. Then the functor $\eta$ is not $\NN^2$-primitive for $T_2$: Take $t = (2m,n) \ge 0$ then if $s=(2m+1,n)$ we have $v \eta^{-1} ( s ) v = \emptyset$ and $s \geq_l t$. Similarly if $t=(2m+1) \ge 0$ then if $s=(2m+2,n)$ we have $v \eta^{-1} (s) v = \emptyset$ and $s \ge_l t$.

\item As in Examples~\ref{ex:S-ud} (iii) let $\eta:T_2\rightarrow\NN^2$ be defined by $\eta(f_1)=(1,0)$ and
$\eta (f_2)=(0,1)$. Then $\eta$ is not $\NN^2$-primitive for $T_2$ as $v\eta^{-1}(m,n)v=\emptyset$ whenever $n>m$.
\end{enumerate}
\end{examples}

\noindent The last two examples above illustrate that upper density and primitivity are unrelated conditions on a $k$-graph. Together they provide a necessary condition for cofinality.

\begin{prop} \label{prop:primcofinal}
Let $\Lambda$ be a $k$-graph with no sources and $\eta : \Lambda \to S$ be a functor where $S$ is a left-reversible semigroup. If $( \Lambda , S, \eta)$ is cofinal then $\eta$ is upper dense.
If $\eta$ is $S$-primitive for $\Lambda$ and upper dense then $( \Lambda , S , \eta )$ is cofinal.
\end{prop}

\begin{proof}
Suppose that $( \Lambda , S, \eta )$ is cofinal. Fix $w \in \Lambda^0$ and $a,b \in S$ and let $v$ be any vertex of $\Lambda$. By cofinality of $(\Lambda, S, \eta)$ there exists $N \in \NN^k$ such that for all $\alpha \in w \Lambda^N$ there is $\beta \in v \Lambda s ( \alpha )$ such that $a \eta ( \beta ) = b \eta ( \alpha )$. Then any element of $b \eta ( w \Lambda^N )$ is of the form
\[
 b \eta ( \alpha ) = a \eta ( \beta ) \geq_l a .
\]

\noindent
Suppose $\eta$ is $S$-primitive and upper dense for $\Lambda$. Since $\eta$ is $S$--primitive for $\Lambda$ there exists $t \in S$ such that for all $v,w \in \Lambda^0$ we have $v \eta^{-1} (s) w \neq \emptyset$ for all $s \geq_l t$. Fix $v,w \in \Lambda^0$ and $a,b \in S$. Since $\eta$ is upper dense there exists $N \in \mathbb{N}^k$ such that $b \eta ( \alpha ) \geq_l at$ for all $\alpha \in w \Lambda^N$. Since $S$ is left-reversible, it is directed, and so by definition $b \eta ( \alpha ) = atu$ for some $u \in S$. But $tu \geq_l t$ and so since $\eta$ is $S$--primitive there exists $\beta \in v \Lambda s ( \alpha )$ such that $\eta ( \beta ) = tu$ and hence $b \eta ( \alpha ) = a \eta ( \beta )$.
\end{proof}

\begin{cor}
Let $\Lambda$ be a row-finite $k$-graph such that $d$ is $\NN^k$ primitive for $\Lambda$ then $( \Lambda , \NN^k , d)$ is cofinal.
\end{cor}

\begin{proof}
Since $d$ is $\NN^k$ primitive for $\Lambda$ it follows that $\Lambda$ has no sources. The result then follows from Lemma~\ref{lem:dud} and Proposition~\ref{prop:primcofinal}.
\end{proof}

\begin{example}
Let $\eta : T_2 \to S$ be any functor, then $\eta (S)$ is a subsemigroup of $S$ since $T_2$ has a single vertex; moreover $\eta$ is $\eta (S)$--primitive for $T_2$. Hence if $\eta$ is upper dense for $T_2$, it follows that $( T_2 , \eta (S) , \eta )$ is cofinal. In particular, in Example\ref{ex:S-ud} (ii) one checks that $( T_2 , \eta ( \NN^2 , \eta )$ is cofinal.
\end{example}

\begin{thm}
Let $\Lambda$ be an aperiodic $k$-graph, $\eta:\Lambda \rightarrow S$ be a functor into a left-reversible semigroup, and $\eta$ be $S$--primitive for $\Lambda$. Then $C^{*}(\Lambda \times_{\eta} S)$ is simple if and only if $\eta$ is upper dense.
\end{thm}

\begin{proof}
If $\eta$ is upper dense then the result follows from Proposition~\ref{prop:primcofinal}. On the other hand if $C^* ( \Lambda \times_\eta S )$ is simple then the result follows from Theorem~\ref{thm:cofinal} and Corollary~\ref{cor:spg1}.
\end{proof}

\section{Skew products by a group} \label{sec:group}

Let $\Lambda$ be a row-finite $k$-graph. A functor $\eta : \Lambda \to G$ defines a coaction $\delta_\eta$ on $C^* (\Lambda )$ determined by $\delta_\eta ( s_\lambda ) = s_\lambda \otimes \eta ( \lambda )$. It is shown in \cite[Theorem 7.1]{PQR} that $C^* ( \Lambda \times_\eta G )$ is isomorphic to $C^* ( \Lambda ) \times_{\delta_\eta} G$. Hence we may relate the simplicity of the $C^*$-algebra of a skew product graph to the simplicity of the associated crossed product. This can be done by using the results of \cite{Q}.

Following \cite[Lemma 7.9]{PQR}, for $g \in G$ the spectral subspace $C^* ( \Lambda )_g$ of the coaction $\delta_\eta$ is given by
\[
C^* ( \Lambda )_g = \overline{\operatorname{span}} \{ s_\lambda s_\mu^* : \eta ( \lambda ) \eta ( \mu )^{-1} = g \} .
\]

\noindent
We define $\sp ( \delta_\eta ) = \{ g \in G : C^* ( \Lambda )_g \neq \emptyset \}$, to be the collection of non-empty spectral subspaces. The \textit{fixed point algebra}, $C^* ( \Lambda )^{\delta_\eta}$ of the coaction is defined to be $C^* ( \Lambda )_{1_G}$. For more details on the coactions of discrete groups on $k$-graph algebras, see \cite[\S 7]{PQR} and \cite{Q}.

We give necessary and sufficient conditions for the skew product graph $C^{*}$-algebra to be simple in terms of the fixed-point algebra as our main result in Theorem~\ref{thm:coact}. We are particularly interested in the case when $\eta$ is the degree functor.

\begin{dfn}
Let $\Lambda$ be a row-finite $k$-graph, $G$ be a discrete group and $\eta : \Lambda \to G$ a functor, then we define
\[
\Gamma ( \eta ) = \{ g \in G : g = \eta ( \lambda ) \eta ( \mu )^{-1} \text{ for some } \lambda , \mu \in \Lambda \text{ with } s ( \lambda ) = s ( \mu ) \}.
\]
\end{dfn}

\begin{lem} \label{lem:gammaeta}
Let $\Lambda$ be a row-finite graph with no sources and $\eta : \Lambda \to G$ a functor, where $G$ is a discrete group.
\begin{enumerate}[(a)]
\item If $( \Lambda , G , \eta )$ is cofinal then $\Gamma ( \eta ) = G$.
\item $\sp ( \delta_\eta ) = G$ if and only if $\Gamma ( \eta ) = G$.
\end{enumerate}
\end{lem}

\begin{proof}
Fix $g \in G$ and write $g=b^{-1}a$ for some $a,b \in G$. Now fix $v,w \in \Lambda^0$; since $( \Lambda , G , \eta )$ is cofinal there exist $\lambda , \mu \in \Lambda$ with $s ( \lambda ) = s ( \mu )$ such that $a \eta ( \mu ) = b \eta ( \lambda )$. Hence $b^{-1} a = \eta ( \lambda ) \eta ( \mu )^{-1}$ and so $g \in \Gamma ( \eta )$. Since $g$ was arbitrary the result follows.

The second statement follows by definition.
\end{proof}

\begin{thm} \label{thm:coact}
Let $\Lambda$ be an aperiodic row-finite $k$-graph with no sources, $\eta : \Lambda \to G$ a functor and $\delta_\eta$ the associated coaction of $G$ on $C^* ( \Lambda )$. Then $C^* ( \Lambda \times_\eta G )$ is simple if and only if $C^* ( \Lambda )^{\delta_\eta}$ is simple and $\Gamma ( \eta ) = G$.
\end{thm}

\begin{proof}
By \cite[Theorem 7.1]{PQR} it follows that $C^* ( \Lambda \times_\eta G )$ is isomorphic to $C^* ( \Lambda ) \times_{\delta_\eta} G$. Then by \cite[Theorem 2.10]{Q} $C^* ( \Lambda ) \times_{\delta_\eta} G$ is simple if and only if $C^* ( \Lambda )^{\delta_\eta}$ is simple and $\operatorname{sp} ( \delta_\eta ) =G$. The result now follows from Lemma~\ref{lem:gammaeta}.
\end{proof}

\begin{example} \label{ex:dcase}
Let $\Lambda$ be a row-finite $k$-graph with no sources and $d : \Lambda \to \NN^k$ be the degree functor. We claim that $\Gamma (d) = \ZZ^k$. Fix $p \in \ZZ^k$, and write $p=m-n$ where $m,n \in \NN^k$. Since $\Lambda$ has no sources, for every $v \in \Lambda^0$ there is $\lambda \in \Lambda^m v$ and $\mu \in \Lambda^n v$. Then
\[
d ( \lambda ) - d ( \mu ) = m - n = p \in \Gamma ( d ),
\]

\noindent and so $\Gamma (d) = \ZZ^k$. Since $\Gamma(d)=\ZZ^{k}$, and $(\Lambda, \ZZ^{k}, d)$ is aperiodic, we have that $C^* ( \Lambda )^{\delta_d}$ is simple if and only $( \Lambda , \NN^k, d )$ is cofinal.
\end{example}

\noindent We seek conditions on $\Lambda$ that will guarantee $( \Lambda , \NN^k ,d )$ is cofinal.

\section{The gauge coaction}

The coaction $\delta_d$ of $\mathbb{Z}^k$ on $C^* ( \Lambda )$ defined in Section~\ref{sec:group} is such that
the fixed point algebra $C^* ( \Lambda )^{\delta_d}$ is precisely the fixed point algebra $C^* ( \Lambda )^\gamma$ for the canonical gauge action of $\mathbb{T}^k$ on $C^* ( \Lambda )$ by the Fourier transform (cf.\ \cite[Corollary 4.9]{C}.

By \cite[Lemma 3.3]{KP} the fixed point algebra $C^* ( \Lambda )^\gamma$ is AF, and is usually referred to as the AF core. In Theorem~\ref{thm:fpa} we use the results of the last two sections to give necessary and sufficient conditions for the AF core $C^* ( \Lambda )^\gamma$ to be simple when $\Lambda^0$ is finite.
When there are infinitely many vertices we show, in Theorem~\ref{thm:notcofinal} that in many
cases the AF core is not simple.

The AF core of a $k$-graph algebra plays a significant role in the development of crossed products by endomorphisms. Results of Takehana and Katayama \cite{KT} show that when $\Lambda$ is a finite $1$-graph such that the core $C^* ( \Lambda )$ is simple, then every nontrivial automorphism of $C^* ( \Lambda )$ is outer (see \cite[Proposition 3.4]{PRW}).

We saw in Example~\ref{ex:spgnotconn} that a $k$-graph being strongly connected is not enough to guarantee that $\Lambda \times_d \ZZ^k$ is cofinal, and hence by \cite[Theorem 3.1]{RobS} $C^* ( \Lambda \times_d \ZZ^k )$ is not simple and then by Theorem~\ref{thm:coact} the AF core is not simple. Another condition is required to guarantee that $\Lambda \times_d \ZZ^k$ is cofinal, which is suggested by \cite{PRho} and was introduced in Section \ref{sec:primitivity}:

\begin{thm} \label{thm:cofinal-prim}
Let $\Lambda$ be a row-finite $k$-graph with no sinks and sources and $\Lambda^0$ finite. If $( \Lambda , d , \mathbb{Z}^k )$ is cofinal then $\Lambda$ is primitive.
\end{thm}

\begin{proof}
 We claim that for $v \in
\Lambda^0$ there is
$N(v) \in \NN^k$ such that for all $n \ge N(v)$ we have $v \Lambda^n v \neq
\emptyset$. Fix $(v,0) \in ( \Lambda \times_d \ZZ^k )^0$ then for each $w \in
\Lambda^0$, when we apply the cofinality condition to $(w,0) \in ( \Lambda
\times_d \ZZ^k )^0$ we obtain $N_w \in \NN^k$ such that $(v,0) ( \Lambda
\times_d \ZZ^k ) s ( \alpha , 0 ) \neq \emptyset$ for all $( \alpha , 0 ) \in (
w, 0 ) ( \Lambda \times_d \ZZ^k )^{N_w}$. Define
$N=\max_{w\in\Lambda^{0}}\{N_{w}\}$, which is finite since $\Lambda^{0}$ is
finite.

By Proposition \ref{prop:cofinal-SC} it follows that $\Lambda$ is strongly
connected,
hence there exists $\alpha\in v\Lambda v$ with $d(\alpha)=r>0$. Hence, there
exists $t\geq 1$ such that $tr\geq N$. Let $N(v)=tr$.

Let $m=n-tr\geq0$. Since $\Lambda$ has no sources, $v\Lambda^{m}\neq\emptyset$;
hence
there exists $\gamma\in v\Lambda^{m}$. Let $w=s(\gamma)$. For
$(v,0),(w,0)\in(\Lambda\times_{d}\Z^{k})^{0}$, we have $( \alpha^t , 0)
\in(v,0)(\Lambda\times_{d}\Z^{k})^{tr}$ where $tr\geq N\geq N_{w}$. By
cofinality and Lemma \ref{lem:cofinalconn} (b), there exists
$(\beta,0)\in(w,0)(\Lambda\times_{d}\Z^{k}) (v,tr)$ as $s( \alpha^t , 0) =
(v,tr)$. As $\beta \in w \Lambda^{tr} v$ it follows that $\gamma\beta\in
v\Lambda^{n}v$, which proves the claim.
\end{proof}


\noindent The following result generalises results from \cite{PRho}:

\begin{thm} \label{thm:fpa}
Let $( \Lambda , d)$ be a row-finite $k$-graph with no sinks or sources, and $\Lambda^0$ finite. Then $C^{*}(\Lambda)^{\delta_d}$ is simple if and only if $\Lambda$ is primitive.
\end{thm}

\begin{proof}
Suppose that $\Lambda$ is primitive. Then $( \Lambda , \ZZ^k , d)$ is strongly connected and cofinal by Remarks~\ref{rmks:noss}. Hence $C^* ( \Lambda \times_d \ZZ^k )$ is simple and so $C^* ( \Lambda )^{\delta_d}$ is simple by Theorem~\ref{thm:coact}.

Suppose that $C^* ( \Lambda )^{\delta_d}$ is simple. Recall from Example~\ref{ex:dcase} that since $\Lambda$ has no sources
then $\Gamma ( d ) = \ZZ^k$. Then by Theorem~\ref{thm:coact}, $C^* ( \Lambda \times_d \ZZ^k )$ is simple, and hence $( \Lambda , d , \ZZ^k )$ is cofinal by \cite[Theorem 3.1]{RobS} and Proposition~\ref{prop:eta-cofinal}. By Theorem~\ref{thm:cofinal-prim} this implies that $\Lambda$ is primitive.
\end{proof}

\begin{example}
Since it has a single vertex it is easy to see that the $2$-graph $\mathbb{F}^2_\theta$ defined in Examples \ref{ex:kex} (d) is primitive. Hence by Theorem~\ref{thm:fpa} we see that $C^* ( \mathbb{F}^2_\theta )^\gamma$ is simple for all $\theta$. Indeed in \cite[\S 2.1]{DY} it is shown that $C^* ( \mathbb{F}^2_\theta )^\gamma \cong \operatorname{UHF} (mn)^\infty$.
\end{example}

\noindent We now turn our attention to the case when $\Lambda^0$ is infinite.
We adapt the technique used in \cite{PRho} to show that, in many cases the AF core is not simple.

\begin{dfn}
Let $\Lambda$ be a row-finite $k$-graph with no sources. For $v \in \Lambda^0$, $n \in \NN^k$ let
\begin{align*}
V(n,v) &= \{ s ( \lambda ) : \lambda \in v \Lambda^m, m \le n \} \\
FV(n,v) &= V (n,v) \backslash \cup_{i=1}^k V(n-e_i,v) .
\end{align*}
\end{dfn}

\begin{rmks} \label{rmks:useful}
For $v \in \Lambda^0$, $m \le n \in \NN^k$ we have, by definition, that $V(m,v) \subseteq V(n,v)$.

For $v \in \Lambda^0$, $n \in \NN^k$ the set $FV(n,v)$ denotes those vertices
which connect to $v$ with a path of degree $n$ and there is no path from that vertex to $v$ with degree less than $n$.
\end{rmks}

\begin{lem} \label{lem:nogrow}
Let $\Lambda$ be a row-finite $k$-graph with no sources. For $v \in \Lambda^0$, $n \in \NN^k$ then $V(n,v)$ is finite and if $V(n)=V(n-e_i)$ for some $1\le i \le k$ then $V(n+re_i)=V(n-e_i)$ for all $r \ge 0$.
\end{lem}

\begin{proof}
Fix, $v \in \Lambda^0$, $n \in \NN^k$, since $\Lambda$ row-finite it follows that $\cup_{m\le n} v \Lambda^m$ is finite and hence so is $V(n,v)$.

Suppose, without loss of generality that $V(n)=V(n-e_1)$. Let $w \in V(n+e_1)$, then there is $\lambda \in v \Lambda^{n+e_1} w$. Now $\lambda ( 0,n) \in v \Lambda^n$ and so $s ( \lambda (0,n) ) \in V(n) = V(n-e_1)$. Hence there is $\mu \in v \Lambda^m s ( \lambda(0,n) )$ for some $m \le n - e_1$ and so $\mu \lambda (n,n+e_1) \in v \Lambda^{m+e_1}$. Since $s ( \mu \lambda (n,e+e_1) ) = s ( \lambda ) = w$ and $m+e_1 \le n$ it follows that $w \in V(n)$. As $w$ was an arbitrary element of $V(n+e_1)$ it follows that $V(n+e_1) \subseteq V(n) = V(n-e_1)$.
By Remarks~\ref{rmks:useful} we have $V(n-e_1 ) \subseteq V(n+e_1)$ and so
$V(n+e_1)=V(n-e_1)$. It then follows that $V(n+re_1) = V(n-e_1)$ for $r \ge 0$ by an elementary induction argument.
\end{proof}

\noindent
We adopt the following notation, used in \cite{KPS1}: Let $\Lambda$ be a $k$-graph
for $1 \le i \le k$ we set $\Lambda^{\NN e_i} = \cup_{r \ge 0} \Lambda^{r e_i}$.

\begin{prop} \label{prop:doesgrow}
Let $\Lambda$ be a row-finite $k$-graph with no sources such that for all $w \in \Lambda^0$ and for $1 \le i \le k$, the set $s^{-1} \left( w \Lambda^{\NN e_i} \right)$ is infinite. Then for all $n \in \NN^k$, $v \in \Lambda^0$ we have $FV(n,v) \neq \emptyset$.
\end{prop}

\begin{proof}
Suppose, for contradiction, that $FV(n,v) = \emptyset$ for some $n \in \NN^k$ and $v \in \Lambda^0$. Then, without loss of generality we may assume that $V(n)=V(e-e_1)$.

Let $\lambda \in v \Lambda^n$, then $s ( \lambda ) \in V(n) = V(n-e_1)$. Fix $r \ge 0$, then since $\Lambda$ has no sources there is $\mu \in s ( \lambda ) \Lambda^{r e_1}$. Then $\lambda \mu \in v \Lambda^{n+re_1}$ and so $s ( \lambda \mu ) = s ( \mu ) \in V( n+re_1,v)$. By Lemma ~\ref{lem:nogrow} it follows that $V( n + r e_1 ) = V ( n - e_1 )$ and so for any $\mu \in s ( \lambda ) \Lambda^{\NN e_1}$ we have
$s ( \mu ) \in V (n-e_1)$. By Remarks~\ref{rmks:useful} $V ( n - e_1 )$ is finite
and so we have contradicted the hypothesis that $s^{-1} \left( w \Lambda^{\NN e_1} \right)$ is infinite.
\end{proof}

\noindent Note that $k$-graphs satisfying the hypothesis of Proposition~\ref{prop:doesgrow} must have infinitely many vertices.
\noindent The following result generalises results from \cite{PRho}:

\begin{thm} \label{thm:notcofinal}
Let $\Lambda$ be a row-finite $k$-graph with no sources such that for all $w \in \Lambda^0$ and for $1 \le i \le k$, the set $s^{-1} \left( w \Lambda^{\NN e_i} \right)$ is infinite. Then $\Lambda \times_d \ZZ^k$ is not cofinal.
\end{thm}

\begin{proof}
Suppose, for contradiction, that $\Lambda \times_d \ZZ^k$ is cofinal.

Fix $v \in \Lambda^0$ then since $\Lambda$ is row-finite and has no sources $W = s^{-1} \left( v \Lambda^{e_1} \right)$ is finite and nonempty. Without loss of generality let $W = \{ w_1 , \ldots , w_n \}$.

Since $\Lambda \times_d \ZZ^k$ is cofinal, for $1 \le i \le n$ if we consider $(w_i,0)$ and $(v,0) \in \Lambda^0 \times \ZZ^k$ then there is $N_i \in \NN^k$ such that for all $( \alpha , 0 ) \in (w_i,0) \left( \Lambda \times_d \ZZ^k \right)^{N_i}$ we have $(v,0) \left( \Lambda \times_d \ZZ^k \right) ( s ( \alpha ) ,N_i ) \neq \emptyset$. Let $N = \max \{ N_1 , \ldots , N_n \}$. By Proposition~\ref{prop:doesgrow} $FV(N+e_1,v)\neq \emptyset$, hence there is $\lambda \in v \Lambda^{N+e_1}$ such that there is no path of degree less than $N+e_1$ from $s(\lambda)$ to $v$. Without loss of generality $s (\lambda(0,e_1))=w_1$, and so $( \lambda (e_1,N+e_1) , 0 ) \in (w_1,0) \left( \Lambda \times_d \ZZ^k \right)^N$. Since $N \ge N_1$ and $\Lambda$ has no sources, by Lemma~\ref{lem:cofinalconn}(ii) there is $( \alpha , 0 ) \in (v,0) \left( \Lambda \times_d \ZZ^k \right) ( s ( \lambda ), N)$ which implies that $\alpha \in v \Lambda^N s ( \lambda )$, contradicting the defining property of $\lambda \in v \Lambda^{N+e_1}$.
\end{proof}

\begin{examples}
\begin{enumerate}
\item Let $\Lambda$ be a strongly connected $k$-graph with $\Lambda^0$ infinite, then $\Lambda$ has no sources and for all $w \in \Lambda^0$ we have $s^{-1} \left( w \Lambda^{\NN e_i} \right)$ is infinite for $1 \le i \le k$. Hence by Theorem~\ref{thm:notcofinal} it follows that $\Lambda \times_d \ZZ^k$ is not cofinal.
\item Let $\Lambda$ be a $k$-graph with $\Lambda^0$ infinite, no sources and no paths with the same source and range. Then for all $w \in \Lambda^0$ we have $s^{-1} \left( w \Lambda^{\NN e_i} \right)$ is infinite for $1 \le i \le k$. Hence by Theorem~\ref{thm:notcofinal} it follows that $\Lambda \times_d \ZZ^k$ is not cofinal.
\end{enumerate}
\end{examples}


\begin{thebibliography}{kp3}


\bibitem{ALRS} A. an Huef, M. Laca, I. Raeburn and A. Sims, \textit{KMS states on $C^*$-algebras associated to higher rank graphs}, 
 J. Math. Anal. Appl., \textbf{405} (2013), 388--399. 

\bibitem{C} T.\ Crisp. \textit{Corners of graph algebras}, J. Operator Theory, \textbf{60} (2008), 253--271.

\bibitem{DPY} K. R. Davidson, S. C. Power and D. Yang, \textit{Atomic representations of rank 2 graph algebras}, J. Funct. Anal., \textbf{255} (2008), 819--853.

\bibitem{DY} K. R.\ Davidson and D.\ Yang, \textit{Periodicity in rank 2 graph algebras}, Canad. J. Math., {\bf 61} (2009), 1239--1261.

\bibitem{E} D. G. Evans, \textit{On the K-theory of higher rank graph C*-algebras}, New York J. Math., \textbf{14} (2008), 1--31.

\bibitem{HRSW} {R. Hazlewood, I. Raeburn, A. Sims and S. Webster}, \textit{Remarks on some fundamental results about higher-rank graphs and their $C^*$-algebras}, Proc. Edinburgh Math. Soc., {\bf 56} (2013), 575--597.

\bibitem{KaP} S.\ Kang and D.\ Pask. \textit{Aperiodicity and the primitive ideal space of a row-finite $k$-graph $C^*$-algebra}, {\tt arXiv:math/1105.1208 [math.OA].}

\bibitem {KT} Y. Katayama and H. Takehana, \textit{On automorphisms of generalized Cuntz algebras}, Internat. J. Math., \textbf{9} (1998), 493--512.

\bibitem{KP} A.\ Kumjian and D.\ Pask, \textit{Higher Rank Graph $C^{*}$-algebras}, New York J. Math., \textbf{6} (2000), 1--20.

\bibitem{kp3} A. Kumjian and D. Pask, \textit{Actions of ${\bf Z}^k$ associated to higher rank graphs}, Ergod. Th. \& Dynam. Sys., \textbf{23} (2003), 1153--1172.

\bibitem{KPS1} A. Kumjian, D. Pask and A. Sims, \textit{$C^*$-algebras associated to covering of $k$-graphs}, Doc. Math., \textbf{13} (2008), 161--205.


\bibitem{LewinSims} P. Lewin and A. Sims, \textit{Aperiodicity and Cofinality for Finitely Aligned Higher-Rank Graphs}, Math. Proc. Cambridge Philosophical Soc., \textbf{149} (2010), 333--350.

\bibitem{MPR} B.\ Maloney, D.\ Pask and I.\ Raeburn, \textit{Skew products of higher-rank graphs and crossed products by semigroups}, Semigroup Forum (to appear).

\bibitem{PQR} D.\ Pask, J.\ Quigg and I.\ Raeburn, \textit{Coverings of $k$-graphs}, J. Alg., \textbf{289} (2005), 161--191.

\bibitem{PRY} D. Pask, I. Raeburn and T. Yeend, \textit{Actions of semigroups on directed graphs and their $C^*$-algebras}, J. Pure Appl. Algebra, \textbf{159} (2001), 297--313.

\bibitem{pr1} D.\ Pask and I.\ Raeburn, \textit{On the $K$-Theory of Cuntz-Krieger algebras}, Publ.\ RIMS Kyoto, \textbf{32} (1996), 415--443.

\bibitem{PRW} D. Pask, I. Raeburn and N. Weaver, \textit{Periodic $2$-graph arising from subshifts}, Bull. Aust. Math. Soc., \textbf{82} (2010), 120--138.

\bibitem{PRho} D.\ Pask and S-J.\ Rho. \textit{Some intrinsic properties of simple graph $C^*$-algebras}, Operator Algebras and Mathematical Physics, Constanza, Romania, 2001, Theta Foundation, \textbf{2003}, 325--340.

\bibitem{Po} S.\ Power, \textit{Classifying higher rank analytic Toeplitz algebras}, New York J. Math., {\bf 13} (2007), 271--298.

\bibitem{Q} J.\ Quigg \textit{Discrete $C^*$--coactions and $C^*$--algebraic bundles.} J.\ Austral.\ Math.\ Soc., \textbf{60} (1996), 204--221.

\bibitem{RSY} I.\ Raeburn, A.\ Sims and T.\ Yeend, \textit{Higher-rank graphs and their $C^{*}$-algebras}, Proc. Edinburgh Math. Soc., \textbf{46} (2003), 99--115.

\bibitem{RSY2} I.\ Raeburn, A.\ Sims and T.\ Yeend, \textit{The $C^{*}$-algebras of finitely aligned higher-rank graphs} J. Funct. Anal., \textbf{213} (2004), 206--240.

\bibitem{RobS} D. I.\ Robertson and A.\ Sims, \textit{Simplicity of $C^*$-algebras associated to higher rank graphs}, Bull. London Math. Soc., \textbf{39} (2007), 337--344.

\bibitem{RobertsonSteger} G.\ Robertson, T.\ Steger, \textit{Affine buildings, tiling systems and higher rank Cuntz-Krieger algebras}, J. reine angew. Math., \textbf{513} (1999), 115--144.

\end{thebibliography}
\end{document}